\documentclass[a4paper,12pt]{article}
\usepackage{amssymb}
\usepackage{amsmath}
\usepackage{dsfont}
\usepackage{graphicx}
\usepackage[all]{xy}
\usepackage{a4wide}
\usepackage{xcolor}
\usepackage{hyperref}
\numberwithin{equation}{section}

\renewcommand{\phi}{\varphi}

\def\and{{\rm and}}

\usepackage{amsthm}

\newtheorem{lemma}{Lemma}[section]  
\newtheorem{theorem}[lemma]{Theorem} 

\newtheorem{proposition}[lemma]{Proposition}

\theoremstyle{remark}

\theoremstyle{definition}

\newtheorem{necs-cond}[lemma]{Necessary Condition}

\CompileMatrices 

\begin{document}
\pagenumbering{gobble}

\begin{center}

\textsc{\textbf{TEL AVIV UNIVERSITY}} \vspace{8mm}

\textsc{THE RAYMOND AND BEVERLY SACKLER} \vspace{2mm}

\textsc{FACULTY OF EXACT SCIENCES} \vspace{2mm}

\textsc{SCHOOL OF MATHEMATICAL SCIENCES} \vspace{14mm}

\LARGE{IRREDUCIBLE POLYNOMIALS WITH VARYING CONSTRAINTS ON COEFFICIENTS}

\vspace{10mm}

\small{Thesis submitted in partial fulfilment of the requirements for the M.Sc. \\ degree in the School of Mathematical Sciences, Tel Aviv University} 

\vspace{14mm}

\small{By} \vspace{2mm}

\normalsize{\textbf{Eyal Moses}}

\vspace{10mm}

\small{The research work for this thesis has been carried out \\ under the supervision of} \vspace{3mm}

\large{Prof. Lior Bary-Soroker} \vspace{30mm}

\small{July 2017}

\end{center}
\newpage 
\pagenumbering{arabic}

\textbf{{\LARGE Acknowledgements}} \vspace{3mm}

I thank my advisor, Lior Bary-Soroker, for his tremendous support and patience throughout this research. He taught me the importance of seeing concrete results in context of er scheme of mathematical knowledge. 
I am grateful to Lior 
for setting a fine example of how mathematical research should be performed. 

\newpage

\tableofcontents

\newpage


%

\begin{abstract}
We study the number of prime polynomials of degree $n$ over $\mathbb{F}_q$ in which the $i^{th}$ coefficient is either preassigned to be $a_i \in \mathbb{F}_q$ or outside a small set $S_i \subset \mathbb{F}_q$. This serves as a function field analogue of a recent work of Maynard, which counts integer primes that do not have specific digits in their base-$q$ expansion. Our work relates to Pollack's and Ha's work, which count the amount of prime polynomials with $\ll \sqrt{n}$ and $\ll n$ preassigned coefficients, respectively. Our result demonstrates how one can prove asymptotics of the number of prime polynomials with different types of constraints to each coefficient.
\end{abstract}

\section{Introduction}

The number of prime polynomials $P$ of degree $n$ over a finite field $\mathbb{F}_q$ has been known to be approximately $q^n/n$ for quite some time. A lot of work has been done on finding the distribution of primes that satisfy a particular condition. Writing
\begin{equation*}
	P = T^n + \sum_{i = 0}^{n-1}b_iT^i,
\end{equation*}
we might ask how many primes $P$ exist for which the tuple $(b_0,b_1...,b_{n-1})$ satisfies given conditions. A natural condition is one of the form $b_i = a$ for some $0 \leq i \leq n-1$, $a \in \mathbb{F}_q$. More generally, let $\mathcal{I} \subset \{0,\dots,n-1\}$ be a set of indices, and denote $I = \# \mathcal{I}$. Let $\{a_i\}_{i \in \mathcal{I}}$ be corresponding coefficients in $\mathbb{F}_q$, satisfying $a_0 \neq 0$ if $0 \in \mathcal{I}$. Denote by $\mathcal{A}$ the set of all monic, degree $n$ polynomials such that the coefficient of $T^i$ is $a_i$ for all $i \in \mathcal{I}$:
\begin{equation}\label{pollack_form}
\mathcal{A} = \left\{T^n + \sum_{j = 0}^{n-1}b_jT^j \in \mathbb{F}_q[T] : b_i = a_i~ \forall i \in \mathcal{I} \right\}
\end{equation}
 A topic of investigation has been to understand the distribution of prime polynomials in $\mathcal{A}$.
Hansen and Mullen conjectured that whenever $n \geq 3$ and $\# \mathcal{I} = 1$, the set $\mathcal{A}$ contains prime polynomials. This conjecture was proven by Wan \cite{Wan97} if $n$ or $q$ are sufficiently large, and the remaining cases were later solved by Ham and Mullen in \cite{HM98}. 

Further work went into finding the asymptotic behavior of the number of primes in $\mathcal{A}$. One might expect the number of primes in the set $\mathcal{A}$ to be $\frac{1}{q^I}\pi_q(n)$ in case $0 \notin \mathcal{I}$, and $\frac{1}{q^I(q-1)}\pi_q(n)$ if $0 \in \mathcal{I}$. Pollack \cite{Po} proves that the asymptotics are indeed so, provided that $\# \mathcal{I} = I < c \sqrt{n}$ with $c < 1$. Further progress has been made by Ha in \cite{junsoo}, where existence of primes in $\mathcal{A}$ is proven given that $q$ is sufficiently large with respect to $\#\mathcal{I}/n$. We prove a weaker result than in  \cite{junsoo}, which generalizes Pollack's theorem in \cite{Po}:
\newpage
\begin{theorem}\label{pure_pollack_case}
Given $\mathcal{A}, \mathcal{I}, \{a_i\}_{i \in \mathcal{I}} \subset \mathbb{F}_q$ as before, write $I = \#\mathcal{I}$, $m_{n,I} = \min\{n/I, \sqrt{n}\}$, and $\rho = I / n$. Denote $\mathfrak{S}_\mathcal{I} = q^{-I}$ if $0 \notin \mathcal{I}$, and $\mathfrak{S}_\mathcal{I} = q^{-(I-1)}(q-1)^{-1}$ otherwise.
If $I = o\left(n/\log(n)\right)$, then
\begin{equation*}
\left| \left(\sum_{P \in \mathcal{A}} 1\right)  - \mathfrak{S}_\mathcal{I} \cdot \pi_q(n) \right| ~\leq ~ \left(1 + o(1)\right)q^{n - \frac{1}{2}\lfloor \frac{n}{2} \rfloor} +q^{n  -I}q^{-\left(1 + o(1)\right)m_{n,I}} ~.
\end{equation*}
If $I > 2\sqrt{n}$, the following bound holds:
\begin{equation*}\left| \left(\sum_{P \in \mathcal{A}} 1\right)  - \mathfrak{S}_\mathcal{I} \cdot \pi_q(n) \right| \leq \left(1 + o(1)\right)q^{n - \frac{1}{2}\lfloor \frac{n}{2} \rfloor} + q^{n-I}q^{-n/I + 4 + B_{q,\rho}}~,
\end{equation*} 
with $B_{q,\rho}$ tending to zero as $q$ grows to infinity provided that $n$ is sufficiently large in terms of $\rho$.

\end{theorem}

Another natural question one might ask is how many primes satisfy the condition $b_i \neq a$ for all $0 \leq i < n$, given $a \in \mathbb{F}_q$. A surprising result by Maynard \cite{Ma16} shows that the number of rational primes without a specific digit in its decimal expansion is of the correct asymptotic. Maynard further proves that the correct asymptotic is kept when taking coefficients in the $q$-basis outside of a set $S \subset \{0, \dots, q-1\}$ of size $\#S < q^{23/80}$. We adopt the method of Maynard in \cite{Ma16} to obtain an analogous result in function fields. For $a \in \mathbb{F}_q$, denote
\begin{equation}\label{maynard_form}
\mathcal{B} = \left\{ T^n + \sum_{j = 0}^{n-1}b_jT^j \in \mathbb{F}_q[T] : b_j \neq a ~ \forall 0 \leq j < n \right \}
\end{equation} 
The expected number of prime polynomials in $\mathcal{B}$ is $\mathfrak{S}_a \cdot \pi_q(n)$, with $\mathfrak{S}_a  =\frac{(q-1)^{n-1}}{q^{n-1}}$ if $a = 0$, and $ \mathfrak{S}_a  = \frac{(q-1)^{n-1}}{q^{n-1}} \cdot \frac{q - 2}{q-1}$ otherwise.

\begin{theorem}\label{pure_maynard_case}
Let $q \geq 5$, and let $\mathcal{B}$, $a$, and $\mathfrak{S}_a$ be defined as before. Then
\begin{equation*}
\left| \left(\sum_{P \in \mathcal{B}} 1\right)  - \mathfrak{S}_a \cdot \pi_q(n) \right| ~\leq ~ \left(2^n+o(1)\right)q^{n - \frac{1}{2}\lfloor \frac{n}{2} \rfloor} + nq^{n - c\sqrt{n}}  + O(q^{n/2} + 1),
\end{equation*}
with $c = \sqrt{(1-\log_q2)(1 - 2\log_q 2)}$.
\end{theorem}
Note that the result is valid whenever $q \geq 5$, but it is only useful when $2^nq^{n - \frac{1}{2}\lfloor \frac{n}{2} \rfloor} \ll q^n$, i.e. when $q \geq 17$.

In this thesis, we consider sets $\mathcal{C}$ that combine the two constraints, and prove a theorem that generalizes both Theorem \ref{pure_pollack_case} and Theorem \ref{pure_maynard_case}. Let $\mathcal{I} \uplus \mathcal{J}$ be a partition of $\{0,\dots,n-1\}$, and write $I = \# \mathcal{I}$. Let $\{a_i\}_{i \in \mathcal{I}} \subset \mathbb{F}_q$ be such that $a_0 \neq 0$ if $0 \in \mathcal{I}$. Consider sets  $S_i \subset \mathbb{F}_q$ for every $i \in \mathcal{J}$, and write $N_i  = \# S_i$. Moreover, assume that $0 \notin S_0$ if $0 \in \mathcal{J}$. Denote
\begin{equation}\label{general_set}
	\mathcal{C} = \left\{T^n + \sum_{i = 0}^{n-1}b_iT^i : b_i = a_i~ \forall i \in \mathcal{I}, b_j \not \in S_j~ \forall j \in \mathcal{J} \right\}.
\end{equation}
In words, the set $\mathcal{C}$ consists of the monic, degree $n$ polynomials such that for all $i \in \mathcal{I}$ the $i^{th}$ coefficient is prescribed to be $a_i$, and the rest of the coefficients are outside small sets $S_i$.

The number of primes to be expected in $\mathcal{C}$ is $\mathfrak{S} \cdot \pi_q(n)$, with 
\begin{equation}\label{sigma_defn}
\mathfrak{S} = 
\left\{
	\begin{array}{ll}
		 \frac{\prod_{j \in \mathcal{J}} (q - N_j)}{q^{n-1}(q-1)}  & \mbox{if } 0 \in \mathcal{I}\\
		 \\
		 \frac{(q-1-N_0)\prod_{0 < j \in \mathcal{J}} (q - N_j)}{q^{n-1}(q-1)} & \mbox{if } 0 \in \mathcal{J}~.
	\end{array}
\right.
\end{equation}
In section \ref{discussion_section} we give a brief explanation why this is indeed the asymptotic one might expect.
For convenience, we define
\begin{equation}\label{alpha_def}
	\alpha(m) = \sup_{\substack{i_1 < ... < i_m \\ i_j \in \mathcal{J}}}   \prod_{j = 1}^m (N_{i_j} + 1).
\end{equation}
Essentially, $\alpha(m)$ is ``small'' if the averages of all subsets of $\{N_j\}_{j \in \mathcal{J}}$ of size $m$ are ``small". Moreover, 
$\alpha(n-I) = \prod_{j \in \mathcal{J}}(N_j + 1)$.
We are now ready to state our main theorem.

\begin{theorem}\label{main_thm} Let $n \geq 2$, and let $\mathcal{I} \uplus \mathcal{J} = \{0,\dots,n-1\}$. To each $i \in \mathcal{I}$ assign $a_i \in \mathbb{F}_q$, and for every $j \in \mathcal{J}$ assign a set $S_j \subset \mathbb{F}_q$. Denote $I = \# \mathcal{I}$, $N_j = \#S_j$ and let $\alpha, \mathcal{C}$ be defined as in \eqref{alpha_def} and \eqref{general_set}. Assume that for all $j \in \mathcal{J}$, we have \mbox{$N_j < q^\varepsilon$} with $\varepsilon < 1$. Assume further that $\frac{I}{n} < \frac{1}{4} \cdot \frac{1-2\varepsilon}{1-\varepsilon}(1 - \tau)$ for some $\tau > 0$. Denote $s = \sqrt{(1 - \varepsilon)/(1 - 2\varepsilon)}\sqrt{n}$. If $I = o\left(n/\log(n)\right)$, then 
\begin{equation*}
	\left| \left(\sum_{P \in \mathcal{C}} 1\right)  - \mathfrak{S} \cdot \pi_q(n) \right| \leq 
	\left(\alpha(n - I) + o(1)\right)q^{n - \frac{1}{2}\lfloor \frac{n}{2} \rfloor} + q^{n  -I}q^{-\left(1 - 2\varepsilon + o(1)\right)m_{n,I,\varepsilon}} ~,
\end{equation*}
where $P$ ranges only over prime polynomials, with $m_{n,I,\varepsilon} = \min \{n/I,s\}$, and \mbox{$\mathfrak{S}$} is given in \eqref{sigma_defn}.
If $y = \frac{I\cdot s}{n} > 1$ and $n$ is sufficiently large in terms of $\varepsilon$ and $\tau$, the following bound holds:
\begin{equation*}
	\left| \left(\sum_{P \in \mathcal{C}} 1\right)  - \mathfrak{S} \cdot \pi_q(n) \right| \leq  
	\left(\alpha(n - I) + o(1)\right)q^{n - \frac{1}{2}\lfloor \frac{n}{2} \rfloor} + q^{n-I}q^{-(1 - 2\varepsilon)n/I + 4 - 3 \varepsilon + B_{q,\varepsilon,\tau,y}}~,
\end{equation*} 
with $B_{q,\varepsilon,\tau,y}$ tending to zero as $q$ grows to infinity.

\end{theorem}
Note that these bounds are only useful when $\alpha(n-I) < q^{n/4}$, since the first term in the error term is about $q^{-\frac{3}{4}n + \log_q(\alpha(n))}$. For example, as stated before, when $N_i = 1$ for all $0 \leq i < n$,  the result is interesting for $q \geq 17$.
	
Note that Theorem \ref{main_thm} is indeed a generalization of Theorems \ref{pure_pollack_case} and \ref{pure_maynard_case}, since the definition of the set $\mathcal{C}$ in \eqref{general_set} is more general than \eqref{pollack_form} and \eqref{maynard_form}. Indeed, if $S_i = \emptyset$,
then the condition $b_i \notin S_i$ is trivially satisfied. Therefore, if we choose $S_i = \emptyset$ for all $i \in \mathcal{J}$, then $\mathcal{C}$ will be a set in the form of \eqref{pollack_form}. On the other hand, if $\mathcal{I} = \emptyset$, and $S_0 = \dots = S_{n-1} = \{a\}$, then the set $\mathcal{C}$ will be in the form of \eqref{maynard_form}.

\subsection{Notation and definitions.}
Denote by
\begin{equation*}
\mathbb{F}_q(T)_{\infty} = \left\{ \sum_{l<k}a_lT^l : a_l \in \mathbb{F}_q , k \in \mathbb{Z} \right\}~,
\end{equation*}
the completion of $\mathbb{F}_q(T)$ with respect to $1/T$.
We define the unit interval $$\mathcal{U} := \left\{ \sum_{l<0}a_iT^l : a_l \in \mathbb{F}_q \right\} \subset \mathbb{F}_q(T)_{\infty}~.$$
Denote by $\psi(\cdot)$ the additive character on $\mathbb{F}_q$ defined by
\begin{equation}\label{psi_defn}
	\psi(a) = \exp \left(\frac{2\pi i}{p} \mbox{Tr}(a)\right),
\end{equation}
where the trace is taken from $\mathbb{F}_q$ down to its prime field $\mathbb{F}_p$. The Euler totient function is denoted by $\phi$.
We define the map \mbox{$\bold{e}: \mathbb{F}_q(T)_\infty \rightarrow \mathbb{C}$} by
\begin{equation}\label{e_defn}
	\bold{e}\left(\sum_{i = -\infty}^{n} a_iT^i\right) = \psi(a_{-1}).
\end{equation}
We denote the set of monic polynomials of degree $k$ by $\mathcal{M}_k$. We also denote the function field analogue of the usual exponential sum over primes by
\begin{equation}\label{exp_sum_defn}
	f(\theta) := \sum_{P \in \mathcal{M}_n} \bold{e}(\theta P),
\end{equation}
where the sum ranges over the monic irreducible polynomials of degree $n$.
For \\ \mbox{$\beta = \sum_{l < k}\beta_lT^l\in \mathbb{F}_q(T)_\infty$}, denote by $\left\{ \beta \right\} $ the fractional part given by
\begin{equation*}
\left\{ \beta \right\} = \sum_{l < 0} \beta_lT^l \in \mathcal{U}~.
\end{equation*}
We define $K_{x,m} = \{x, \dots, x + m - 1\}$ for $x \geq 0$.
We use $\mathds{1}$ to denote the indicator function, so for a set $A$ we define
\begin{equation}\label{indicator_defn}
\mathds{1}_A(x) = 
\left\{
	\begin{array}{ll}
		1 & \mbox{if } x \in A\\
		 \\
		 0 & \mbox{otherwise}~.
	\end{array}
\right.
\end{equation}
We let $\mathcal{P}$ denote the set of prime polynomials in $\mathbb{F}_q[T]$.

\subsection{Outline of Proof}
Defining the set $\mathcal{C}$ as in \eqref{general_set}, we perform a Fourier transform on the indicator function $\mathds{1}_\mathcal{C} : \mathcal{M}_n \rightarrow \{0, 1\}$:
\begin{equation}\label{fhat_def}
	\hat{F}_{q,n} (\theta) = \sum_{G \in \mathcal{M}_n} \mathds{1}_\mathcal{C}(G)\bold{e}(G\theta).
\end{equation}
The Fourier Inversion Theorem then gives
\begin{equation}\label{fourier_identity}
		\mathds{1}_\mathcal{C}(F) = \frac{1}{q^n}\sum_{G \in \mathcal{M}_n} \hat{F}_{q,n} (T^{-n}G) \bold{e}(-T^{-n}GF),
\end{equation}
and by Parseval's Formula and \eqref{exp_sum_defn}, we get an analytic expression to the prime counting function in Theorem \ref{main_thm}:
\begin{equation}\label{parseval_formula}
	\sum_{P \in \mathcal{C}} 1 = \sum_{G \in \mathcal{M}_n} \mathds{1}_{\mathcal{P}}(G)\mathds{1}_\mathcal{C}(G) = \frac{1}{q^n}\sum_{F \in \mathcal{M}_n}\hat{F}_{q,n}(T^{-n}F)f(-T^{-n}F)~.
\end{equation}
We prove \eqref{fourier_identity} and \eqref{parseval_formula} in the beginning of Section \ref{fhat_bounds_subsec_gen}.

The main term of \eqref{parseval_formula} comes from polynomials $F \in \mathcal{M}_n$ of the form $F = T^n + aT^{n-1}$ with $a \in \mathbb{F}_q$. This is shown in Section \ref{main_term_sub}. For the rest of the polynomials $F \in \mathcal{M}_n$ we bound $\left| f(T^{-n}F) \right|$ and $\left| \hat{F}_{q,n}(T^{-n}F) \right|$. 

Section \ref{preliminaries_section} states circle-method bounds from the literature for $|f|$, which are due to Hayes \cite{Hayes} in the setting of Pollack \cite{Po}. 

Section \ref{fhat_section}, which is the main part of the work, gives bounds for $\left| \hat{F}_{q,n} \right|$.

Section \ref{full_proof} derives the proof of Theorem \ref{main_thm}.

\section{Circle Method Bounds} \label{preliminaries_section}

\begin{lemma}\label{dirichlet} For each $\theta \in \mathcal{U}$, there is a unique pair of coprime polynomials $G, H \in \mathbb{F}_q[T]$ with $H$ monic, $\deg G < \deg H \leq n / 2$, and
	$$ \left|\left\{ \theta - \frac{G}{H} \right\} \right| < \frac{1}{q^{\deg H + n/2}} $$
\end{lemma}
This is an analogue of a well known result of Dirichlet's theorem proven by Hayes in 1966. From now on, we will use the notation $|\theta - G/H|$ as an abbreviation of $\left|\left\{ \theta - G/H \right\} \right|$.

\begin{lemma}\label{circ_method_ext} Let $n \geq 2$. Let $\theta \in \mathbb{F}_q(T)_{\infty}$, and choose $G,H$ as in Lemma \ref{dirichlet}. Then if $1,T \neq H$ is squarefree and $|\theta - G/H| < q^{-n}$, then
\begin{equation*}
	|f(\theta)| \leq  q^{n - \frac{1}{2}\lfloor \frac{n}{2} \rfloor} + q^{n - \deg H}.
\end{equation*}
If $|\theta - G/H| \geq q^{-n}$ or $H$ is not squarefree,
\begin{equation*}
	|f(\theta)| \leq q^{n - \frac{1}{2}\lfloor \frac{n}{2} \rfloor}.
\end{equation*}
\end{lemma}
\begin{proof}
The case where $|\theta - G/H| \geq q^{-n}$ or $H$ is not squarefree is immediate from the statement of Pollack \cite[Lemma 5]{Po}. The assertion of the lemma in the case where $1,T \neq H$ is squarefree and $|\theta - G/H| < q^{-n}$ is proven in Pollack \cite[Lemma 6]{Po}. 
\end{proof}

 \section{Bounds on Fourier Coefficients} \label{fhat_section}

 In this section we give bounds for $\left| \hat{F}_{q,n} \right|$. They are later used in the setting of \eqref{parseval_formula}, where the bounds of $f$ are taken from Lemma~ \ref{circ_method_ext}. From this reason, 
 we need bounds for general $\theta$, and bounds for $\theta$ of the form $\theta = \{ G / H \} \in \mathcal{U}$, with $H \neq 1,T$ squarefree and $G$,$H$ coprime. We call the latter part "fractions". 
 The bounds for general $\theta$ are obtained in Subsection \ref{fhat_bounds_subsec_gen}, 
 while the bounds for $\theta$ that are fractions are obtained in Subsection \ref{fhat_bounds_subsec_spec}. 
 
 \subsection{Auxiliary Results}\label{fhat_auxiliary}
 
 \begin{lemma}\label{deg_zero_bound}
Let $1,T \neq H \in \mathbb{F}_q[T]$ be a monic squarefree polynomial of degree $h$, and \mbox{$G \in \mathbb{F}_q[T]$} coprime to $H$. Write
$\theta =  \sum_{j < 0}\theta_jT^j = \left \{ G/H \right \}$. There is no $i < 0$ such that $\theta_j = 0$ for all $i - h < j \leq i$. 
Equivalently, there are no $h$ consecutive zeros in the coefficients of $\theta$.
\end{lemma}
\begin{proof}
Write $H = T^h + \sum_{j = 0}^{h - 1}h_jT^j$. Since $H \neq 1,T$ is squarefree, there is a polynomial $H_1 \mid H$ such that $\deg H_1 \geq 1$, and $H_1$ is coprime to $T$. Since $G$ is coprime to $H$, $G$ is also coprime to $H_1$. Hence for every $k \in \mathbb{N}$, we have that $T^kG$ is coprime to $H_1$, hence $T^k \theta \notin \mathbb{F}_q[T]$. In other words, for every $k \in \mathbb{N}$ there exists $i < -k$ such that $\theta_i \neq 0$. Assume by way of contradiction that there is $i_0 < 0$ for which \mbox{$\theta_{i_0} = \theta_{i_0-1} = ... = \theta_{i_0 - h + 1} =0$}. Since we know there are infinitely many $i < 0$ for which $\theta_i \neq 0$, we can assume without loss of generality that $\theta_{i_0 - h} \neq 0$. Denote $\tau = H\theta$, and write $\tau = \sum_{i < h}\tau_iT^i$. Observing $\tau_{i_0}$, we get
\begin{equation*}
	\tau_{i_0} = \theta_{i _0- h} + \sum_{j = 0}^{h-1}h_j\theta_{i_0 - j}~.
\end{equation*}
Since $\theta_i = 0$ for all $i - h + 1 \leq i \leq i_0$, we get that $\tau_{i_0} = \theta_{i_0 - h} \neq 0$. This is a contradiction to the choice of $\theta$ which implies that $\tau = H\theta = H \cdot \{G/H\}\in \mathbb{F}_q[T]$.
\end{proof}

\begin{lemma}\label{diff_mult_t}
Take $h \in \mathbb{N}$, $0 \leq x \in \mathbb{Z}$, and $\theta \in \mathcal{U}$. Then there are at most $q$ distinct pairs $G, H \in \mathbb{F}_q[T]$ such that $H \neq 1, T$ is squarefree of degree $h$, $G$ is coprime to $H$ and of smaller degree, and $|T^xG/H - \theta| < q^{-2h}$. Moreover, if $x = 0$, then there is at most one such pair.
\end{lemma}
\begin{proof}
Assume $(G_1,H_1)$ is a pair that satisfies the conditions of the lemma, and define 
\begin{eqnarray*}
H_1' &=& H_1 / \gcd(H_1, T^x)~,  \\
G_1' &=& (G_1T^x / \gcd(H_1, T^x)) \mod H_1'~.
\end{eqnarray*}
Then $\deg H_1' \leq h$, $G_1'$ is coprime to $H_1'$ of smaller degree and $|\theta - G_1'/H_1'| < q^{-2h}$. Thus $(G_1',H_1')$ is the unique pair that corresponds to $\theta$ in the sense of Lemma \ref{dirichlet}. Let $(G_2, H_2)$ be a different pair that satisfies the conditions of the lemma, and define $G_2'$ and $H_2'$ in a similar manner. From the uniqueness property of Lemma \ref{dirichlet}, we have $G_1' = G_2'$ and $H_1' = H_2'$. Since both \mbox{$\deg H_1' = \deg H_2'$} and $\deg H_1 = \deg H_2$, we arrive at $\deg \gcd(H_1, T^x) = \deg \gcd(H_2, T^x)$. From this we know that $\gcd(H_1, T^x) = \gcd(H_2, T^x)$, and $H_1 = H_1' \cdot \gcd(H_1, T^x) = H_2' \cdot \gcd(H_2, T^x) = H_2$. For convenience sake we now denote $H = H_1$, $H' = H_1'$. Since $H$ is squarefree, $T^2 \nmid H$ and thus \mbox{$\gcd(H, T^x) \in \{ 1, T\}$}. We know that $G_1' = G_2'$ but $G_1 \neq G_2$, hence $\gcd(H, T^x) \neq 1$ and thus $\gcd(H, T^x) = T$. This serves as a contradiction when $x = 0$, and thus we have proven the second assertion of the lemma. When $x > 0$, we know that $T^xG_1 \equiv TG_1' \equiv T^xG_2 \mod H$, thus
\begin{equation}
	T^x (G_1 - G_2) \equiv 0 \mod H
\end{equation} 
This means that $H' \mid G_1 - G_2$, but 
\begin{equation}
	\deg(G_1 - G_2) \leq \max\{\deg G_1, \deg G_2\} \leq h - 1 = \deg H'~.
\end{equation}
So $G_1 - G_2 = cH'$ for some $c \in \mathbb{F}_q$. This completes the proof, since there are exactly $q$ polynomials of the form $G_2 = G_1 - cH'$.
\end{proof}

\begin{lemma}\label{keta_exists}
Let $1 \leq m \leq n$ and $\mathcal{I} \subset \{0, \dots, n - 1\}$, and denote $I = \# \mathcal{I}$. For $0 \leq x \leq n - m$, denote $K_{x,m} = \{x, \dots, x + m - 1\}$. Then there exists $0 \leq y \leq n - m$ such that
\begin{equation}\label{keta_for_all_m}
	\#(K_{y,m} \cap \mathcal{I}) < 2m \cdot \frac{I}{n}~.
\end{equation}
Moreover, if $m \leq n/2$, there exists $0 \leq y \leq n-m$ such that
\begin{equation}\label{keta_for_small_m}
	\#(K_{y,m} \cap \mathcal{I}) < \frac{3m}{2} \cdot \frac{I}{n}~.
\end{equation}
\end{lemma}
\begin{proof}
Denote $x_i = m \cdot i$, for $0 \leq i < \lfloor n / m \rfloor$. 
The sets $K_{x_i, m}$ are pairwise disjoint, so 
\begin{equation}\label{max_seg_sum}
	\sum_{i = 0}^{n/m - 1} \#(K_{x_i, m} \cap \mathcal{I}) \leq I~.
\end{equation}
Assume by way of contradiction that $ \#(K_{x_i, m} \cap \mathcal{I}) \geq  \frac{\left(\lfloor n / m \rfloor + 1\right)m}{\lfloor n / m \rfloor} \cdot \frac{I}{n}$ for all $0 \leq i < \lfloor n / m \rfloor$. Since 
\begin{equation*}
	\left(\lfloor n / m \rfloor + 1\right) \cdot m > n~,
\end{equation*}
we get
\begin{equation*}
	\sum_{i = 0}^{n/m - 1} \#(K_{x_i, m} \cap \mathcal{I}) \geq \lfloor n / m \rfloor \cdot \frac{\left(\lfloor n / m \rfloor + 1\right)m}{\lfloor n / m \rfloor} \cdot \frac{I}{n} >
	n \cdot \frac{I}{n} = I,
\end{equation*}
which contradicts \eqref{max_seg_sum}. So there exists $0 \leq i < \lfloor n / m \rfloor$ such that
\begin{equation*}
	\#(K_{x_i, m} \cap \mathcal{I}) <  \frac{\left(\lfloor n / m \rfloor + 1\right)m}{\lfloor n / m \rfloor} \cdot \frac{I}{n}~,
\end{equation*}
from which it is easy to see that there exists $0 \leq y \leq n - m$ that satisfies \eqref{keta_for_all_m} for all $m \leq n$, and that there exists $0 \leq y \leq n - m$ that satisfies \eqref{keta_for_small_m} for all $m \leq n/2$. This completes the proof.
\end{proof}

 \subsection{General Bound}\label{fhat_bounds_subsec_gen}
Recall the definitions of $\alpha$ and $\hat{F}_{q,n}$  given in \eqref{alpha_def} and \eqref{fhat_def}, respectively. Our goal in this subsection is to establish the following bound:
\begin{proposition}\label{gen_bound} Let $n \geq 2$. Then
\begin{equation*}
	\sum_{F \in \mathcal{M}_n} \left|\hat{F}_{q,n}(T^{-n}F)\right| \leq \alpha(n - I)q^n~.
\end{equation*}
\end{proposition}
We start by proving the Fourier Inverse Formula \eqref{fourier_identity} and Parseval's Formula \eqref{parseval_formula}: Recall that
\begin{equation*}
	\hat{F}_{q,n} (\theta) = \sum_{G \in \mathcal{M}_n} \mathds{1}_\mathcal{C}(G)\bold{e}(G\theta).				
\end{equation*}
Developing the right-hand side of \eqref{fourier_identity} gives
\begin{eqnarray*}
	  \frac{1}{q^n}\sum_{G \in \mathcal{M}_n} \hat{F}_{q,n} (T^{-n}G) \bold{e}(-T^{-n}GF) 
	 &=& \frac{1}{q^n}\sum_{G \in \mathcal{M}_n} \sum_{S \in \mathcal{M}_n} \mathds{1}_\mathcal{C}(S)\bold{e}(T^{-n}GS)\bold{e}(-T^{-n}GF) \\
	 &=& \frac{1}{q^n}\sum_{S \in \mathcal{M}_n}\mathds{1}_\mathcal{C}(S) \sum_{G \in \mathcal{M}_n} \bold{e}\left(T^{-n}(F - S)G\right),
\end{eqnarray*}
and by orthogonality relations we get
\begin{eqnarray*}
\frac{1}{q^n}\sum_{G \in \mathcal{M}_n} \hat{F}_{q,n} (T^{-n}G) \bold{e}(-T^{-n}GF)
	 &=& \frac{1}{q^n}\sum_{S \in \mathcal{M}_n}\mathds{1}_\mathcal{C}(S) q^n \mathds{1}_{F = S} 
	 = \mathds{1}_\mathcal{C}(F)~.
\end{eqnarray*}
So we have explicitly shown \eqref{fourier_identity}.

Parseval's Formula in \eqref{parseval_formula} is as easy to derive: by \eqref{fourier_identity} we may substitue \\ $\frac{1}{q^n}\sum_{G \in \mathcal{M}_n} \hat{F}_{q,n} (T^{-n}G) \bold{e}(-T^{-n}GF)$ for $\mathds{1}_\mathcal{C}(F)$, to get
\begin{eqnarray*}
\sum_{P \in \mathcal{C}} 1 &=& \sum_{G \in \mathcal{M}_n} \mathds{1}_{\mathcal{P}}(G)\mathds{1}_\mathcal{C}(G) \\ 
	&=& 
	\sum_{G \in \mathcal{M}_n} \mathds{1}_{\mathcal{P}}(G)
	\left(\frac{1}{q^n}\sum_{F \in \mathcal{M}_n} \hat{F}_{q,n} (T^{-n}F) \bold{e}(-T^{-n}FG)\right).
\end{eqnarray*}
Changing order of summation and noting \eqref{exp_sum_defn} gives
\begin{eqnarray*}
	\sum_{P \in \mathcal{C}} 1 &=& \frac{1}{q^n}\sum_{F \in \mathcal{M}_n} \hat{F}_{q,n} (T^{-n}F) \sum_{G \in \mathcal{M}_n} \mathds{1}_{\mathcal{P}}(G)\bold{e}(-T^{-n}FG)\\
	&=& \frac{1}{q^n}\sum_{F \in \mathcal{M}_n} \hat{F}_{q,n} (T^{-n}F)f(-T^{-n}F),
\end{eqnarray*}
so \eqref{parseval_formula} is established.

Take $\theta = \sum_{l < k} \theta_lT^l \in \mathbb{F}_q(T)_{\infty}$. In order to bound $\left| \hat{F}_{q,n}(\theta) \right|$, we introduce new notation. Define
\begin{equation}\label{z_n_def}
	Z(\theta) =  \{ i \in \mathcal{J} : \theta_{-i-1} = 0\},\quad N(\theta) =  \{ i \in \mathcal{J} : \theta_{-i-1} \neq 0\}~.
\end{equation}
Essentially,  $Z(\theta)$ is the zero set of $\theta$ between $-1$ and $-n$, and $N(\theta)$ is the nonzero set of $\theta$ in the same range. As we can see in Lemma \ref{fhat_bound}, these sets hold most of the information on our bound on $\left| \hat{F}_{q,n}(\theta) \right|$.

\begin{lemma}\label{fhat_bound}
For $\theta = \sum_{l < k}\theta_lT^l$, let $Z(\theta)$, $N(\theta)$ be defined as in \eqref{z_n_def}. Then
\begin{equation*}
	\left| \hat{F}_{q,n}(\theta) \right| \leq  \prod_{i \in N(\theta)} N_i \prod_{i \in Z(\theta)}q~.
\end{equation*}
\end{lemma}
\begin{proof}
Recall the definition of $\mathcal{C}$ in \eqref{general_set}. Using the notation of Theorem \ref{main_thm}, define $\mathcal{C}_i = \{a_i\}$ for~$i \in \mathcal{I}$ and $\mathcal{C}_i = \mathbb{F}_q \backslash S_i$ for $i \in \mathcal{J}$. Define $\mathcal{C}_n = \{1\}$.
We denote the $i^{\mbox{th}}$ coefficient of a polynomial $G \in \mathbb{F}_q[T]$ by $g_i$. Then
\begin{align*}
	\hat{F}_{q,n}(\theta) &= \sum_{G \in \mathcal{M}_n} \mathds{1}_\mathcal{C}(G)\bold{e}(G\theta) = \sum_{G \in \mathcal{M}_n} \prod_{i = 0}^{n} \mathds{1}_{\mathcal{C}_i}(g_i)\bold{e}(T^ig_i \theta) \\
	&= \bold{e}(T^n\theta) \prod_{i = 0}^{n-1} \sum_{g_i \in \mathbb{F}_q}\mathds{1}_{\mathcal{C}_i}(g_i)\bold{e}(T^ig_i\theta)~.
\end{align*}
More explicitly, by the definition of $\bold{e}$ given in \eqref{e_defn} we may write
\begin{equation*}
	\hat{F}_{q,n}(\theta) = \bold{e}(T^n\theta)\prod_{i = 0}^{n-1} \sum_{g_i \in \mathbb{F}_q}\mathds{1}_{\mathcal{C}_i}(g_i)\psi(g_i\theta_{-i-1}),
\end{equation*}
and taking absolute value gives
\begin{equation}\label{fhat_abs}
	\left| \hat{F}_{q,n}(\theta) \right| 
	= \prod_{i = 0}^{n-1} \left| \sum_{g_i \in \mathbb{F}_q}\mathds{1}_{\mathcal{C}_i}(g_i)\psi(g_i\theta_{-i-1})\right|~.
\end{equation}
Denote $X = \sum_{g_i \in \mathbb{F}_q}\mathds{1}_{\mathcal{C}_i}(g_i)\psi(g_i\theta_{-i-1})$. Note that if $i \in \mathcal{I}$, then $|X|=1$. If $i \in \mathcal{J}$, we divide into two cases: If $\theta_{-i-1} = 0$, then the $|X| = q-N_i$. If $\theta_{-i-1} \neq 0$, then when $\mathds{1}_{\mathcal{C}_i}(g_i) = 1$, $g_i$ ranges over $\mathbb{F}_q \backslash S_i$. From orthogonality relations
\begin{equation} \label{zero_case}
X = \sum_{g_i \in \mathbb{F}_q}\mathds{1}_{\mathcal{C}_i}(g_i)\psi(g_i\theta_{-i-1}) = -\sum_{b \in S_i} \psi(b \theta_{-i-1})~,
\end{equation}
thus in this case $ \label{nonzero_case} \left| X\right| \leq N_i$.
Inserting these bounds on $|X|$ into \eqref{fhat_abs} yields
\begin{align*}
	\left| \hat{F}_{q,n}(\theta) \right| 	&\leq  \prod_{i \in N(\theta)} N_i \prod_{i \in Z(\theta)}(q-N_i) \leq   \prod_{i \in N(\theta)} N_i \prod_{i \in Z(\theta)} q~,
\end{align*}
which completes the proof of the lemma.
\end{proof}

\begin{lemma}\label{dont_care_fraction}
For every $\theta, \eta \in \mathbb{F}_q(T)_\infty$ such that $|\theta - \eta| < q^{-n}$, we have \mbox{$\left|\hat{F}_{q,n}(\theta)\right| = \left|\hat{F}_{q,n}(\eta)\right|$}.
\end{lemma}
\begin{proof}
For every $\theta \in \mathbb{F}_q(T)_\infty$, by \eqref{fhat_abs} we have
\begin{equation*}
\left|\hat{F}_{q,n}(\theta)\right| = \prod_{i = 0}^{n-1} \left|   \sum_{g_i \in \mathbb{F}_q}\mathds{1}_{\mathcal{C}_i}(g_i) \psi(g_i\theta_{-i-1})  \right| .
\end{equation*}
From this, it is easy to see that $\left|\hat{F}_{q,n}(\theta)\right|$ depends only on $\theta_{-1}, ..., \theta_{-n}$.
Let $\eta \in \mathbb{F}_q(T)_\infty$ be such that $|\theta - \eta| < q^{-n}$. Then $\theta_i = \eta_i$ for every $-n \leq i \leq -1$, hence $\left|\hat{F}_{q,n}(\theta)\right| = \left| \hat{F}_{q,n}(\eta)\right|$.
\end{proof}

\begin{proof}[Proof of Proposition \ref{gen_bound}]
We turn to prove that
\begin{equation*}
	\sum_{F \in \mathcal{M}_n} \left|\hat{F}_{q,n}(T^{-n}F)\right| \leq \alpha(n - I)q^n~.
\end{equation*}
Note first that $\#\mathcal{J} = n - I$, thus
\begin{equation*}
	\alpha(n-I) = \sup_{\substack{i_1 < ... < i_{n-I} \\ i_j \in \mathcal{J}}}   \prod_{j = 1}^{n-I} (N_{i_j} + 1) =  \prod_{i \in \mathcal{J}} (N_i + 1).
\end{equation*}
Writing $F = T^n + \sum_{i = 0}^{n-1} f_iT^i$ and $\theta_F = T^{-n}F$, define $Z(\theta_F)$, $N(\theta_F)$ as in \eqref{z_n_def}. By \mbox{Lemma \ref{fhat_bound},}
\begin{eqnarray*}
	\sum_{F \in \mathcal{M}_n} \left|\hat{F}_{q,n}(\theta_F)\right| &\leq& \sum_{F \in \mathcal{M}_n} \prod_{i \in N(\theta_F)} N_i \prod_{i \in Z(\theta_F)}q \\
									&=& \sum_{f_0 \in \mathbb{F}_q} \cdots  \sum_{f_{n-1} \in \mathbb{F}_q} 
 \prod_{\substack{i \in \mathcal{J} \\ f_i \neq 0}} N_i \prod_{\substack{i \in \mathcal{J} \\ f_i = 0}} q.
\end{eqnarray*}
Changing order of summation and product, we have
\begin{equation*}
	\sum_{F \in \mathcal{M}_n} \left|\hat{F}_{q,n}(\theta_F)\right| \leq
	\prod_{i \in \mathcal{I}} \left(\sum_{f_i \in \mathbb{F}_q} 1\right) \times 
	\prod_{i \in \mathcal{J}} \left(q + \sum_{0 \neq f_i \in \mathbb{F}_q} N_i\right)~,
\end{equation*}
thus
\begin{eqnarray*}
		\sum_{F \in \mathcal{M}_n} \left|\hat{F}_{q,n}(\theta_F)\right| 
		&\leq& \prod_{i \in \mathcal{I}}q \prod_{i \in \mathcal{J}} \left(q + N_i \cdot (q - 1)\right)
									\leq  q^I\prod_{i \in \mathcal{J}} q(N_i + 1) \\
									&=& q^n\prod_{i \in \mathcal{J}} (N_i + 1)
									= \alpha(n - I)q^n,
\end{eqnarray*}
as claimed.
\end{proof}

 \subsection{Bound for Fractions}\label{fhat_bounds_subsec_spec} Let $n \geq 2$, and $0 \leq h \leq n/2$. Having obtained a bound for $\sum_{F \in \mathcal{M}_n} \left|\hat{F}_{q,n}(T^{-n}F)\right|$, we now turn to bound 
 \begin{equation}\label{y_h_def}
 	Y_h = \sum_{\substack{G,H \\ \deg H = h}} \left|\hat{F}_{q,n}(G/H)\right|~,
\end{equation}
where the sum ranges over $H \neq 1,T$ squarefree, and $G$ coprime to $H$. 
At the end of the section, we incorporate some of the assumptions of Theorem \ref{main_thm} in order to prove
\begin{proposition}\label{sum_over_deg_bound} Let $n \geq 2$, and assume that some $\varepsilon > 0$ satisfies that  $\alpha(m) < q^{\varepsilon m}$ for all $0 < m \leq n$. For every $0 \leq h \leq \min\{n/2,n/I\}$ we have
 \begin{equation}\label{prop_frac_1}
	Y_h \leq q^{n + 3(1 - \varepsilon) - (1 - \varepsilon)n/h + 2\varepsilon h - \varepsilon I},
\end{equation}
 for $h \leq n/2$, we have
 \begin{equation}\label{prop_frac_2}
	Y_h \leq q^{n - I + 1}q^{\left(2\varepsilon + (1 - \varepsilon)4I/n \right) h}~,
\end{equation}
and for $h < n/4$, we have
 \begin{equation}\label{prop_frac_3}
	Y_h \leq  
	q^{n - I + 1}q^{\left(2\varepsilon + (1 - \varepsilon)3I/n \right) h}~.
\end{equation}
\end{proposition}
Note that for $h < n/4$ the bound \eqref{prop_frac_3} is strictly better than \eqref{prop_frac_2}, but sometimes it is more convenient to use  \eqref{prop_frac_2}.
In the following lemma, we do not use specific properties of fractions $G/H$, but instead give a bound to $\left|\hat{F}_{q,n}(\theta)\right|$ that will later be useful when $\theta = G/H$ with $\deg H = h$.
\begin{lemma}\label{bound_fhat_by_h}
Let $1 \leq l \leq n$, \mbox{$\theta \in \mathcal{U}$}, and define $Z(\theta)$, $N(\theta)$ as in  \eqref{z_n_def}. For \mbox{$0 \leq x \leq n - l$}, 
denote $K = K_{x, l} = \{x, \dots, x + l - 1\}$. Define $\overline{K} = \{0 , \dots, n - 1\} \backslash K$ and denote \mbox{$Z_{\notin K}(\theta) = Z(\theta) \cap \overline{K}$}, \mbox{$N_{\notin K} (\theta) = N(\theta) \cap \overline{K}$}. Then
\begin{equation}\label{bound_fhat_by_h_main_eq}
	\left|\hat{F}_{q,n}(\theta)\right|  \leq  \alpha_{\notin K} \left(  \#N_{\notin K}(\theta) \right) q^{\# Z_{\notin K}(\theta)} \left|\hat{F}_{q,l}(T^x\theta) \right|
\end{equation}
with
\begin{equation*}
	\alpha_{\notin K}(m) = \sup_{\substack{i_1 < ...< i_m \\ i_j \in \overline{K}}} \prod_{j = 1}^{m} (N_{i_j}+1)~.
\end{equation*}
\end{lemma}
\begin{proof}
As in \eqref{fhat_abs}, for every $\theta$ we have
\begin{eqnarray*}
\left|\hat{F}_{q,n}(\theta)\right| &=& \prod_{i = 0}^{n-1} \left|\sum_{g_i \in \mathbb{F}_q}\mathds{1}_{\mathcal{C}_i}(g_i)\psi(g_i\theta_{-i-1})\right|
\end{eqnarray*}
Splitting the product into $K$ and $\overline{K}$, we get
\begin{equation*}
	\left|\hat{F}_{q,n}(\theta)\right| = \prod_{i \in K} \left|\sum_{g_i \in \mathbb{F}_q}\mathds{1}_{\mathcal{C}_i}(g_i)\psi(g_i\theta_{-i-1})\right| \cdot \prod_{i \in \overline{K}} \left|\sum_{g_i \in \mathbb{F}_q}\mathds{1}_{\mathcal{C}_i}(g_i)\psi(g_i\theta_{-i-1})\right|~.
\end{equation*}
Since $K = K_{x, l}$, the left-hand element of the product is exactly $ \left|\hat{F}_{q,l}(T^x\theta)\right|$. We use similar arguments to those of Lemma \ref{fhat_bound} in order to bound $\prod_{i \in \overline{K}} \left|\sum_{g_i \in \mathbb{F}_q}\mathds{1}_{\mathcal{C}_i}(g_i)\psi(g_i\theta_{-i-1})\right|$. 
This gives
\begin{equation} \label{fhat_div_bound}
	 \left|\hat{F}_{q,n}(\theta)\right| \leq \left|\hat{F}_{q,{l}}(T^x\theta)\right| \prod_{i \in Z(\theta) \cap \overline{K}} q \prod_{i \in N(\theta) \cap \overline{K}} N_i~.
\end{equation}
Considering the definition of $Z_{\notin K}(\theta)$ and $N_{\notin K} (\theta)$,  \eqref{fhat_div_bound}  translates to
\begin{equation*}
	 \left|\hat{F}_{q,n}(\theta)\right| \leq \left|\hat{F}_{q,l}(T^x\theta)\right| \prod_{i \in Z_{\notin K}(\theta)} q \prod_{i \in N_{\notin K}(\theta)} N_i~.
\end{equation*}
Considering the definition of $\alpha_{\notin K}(m)$, we see that
\begin{align*} \label{alpha_2h_ineq}
	\left|\hat{F}_{q,n}(\theta)\right| 
	 &\leq \alpha_{\notin K} \left(  \#N_{\notin K}(\theta) \right) q^{\# Z_{\notin K}(\theta)} \left|\hat{F}_{q,l}(T^x\theta) \right|.  \nonumber 
\end{align*}
This completes the proof.
\end{proof}


\begin{lemma}\label{bound_fhat_by_z_bound}
Let $1 \leq l \leq n$, $0 \leq x \leq n - l$, $\theta \in \mathcal{U}$. Denote $K = K_{x,l}$, and let $\overline{K}$, $Z_{\notin K}(\theta)$, $N_{\notin K}(\theta)$, and $\alpha_{\notin K}$ be defined as in Lemma \ref{bound_fhat_by_h}. Write $I_{\in K} =  \#\left(\mathcal{I} \cap K \right)$, $I_{\notin K} = I - I_{\in K}$. For an integer $t$ in the range $\# Z_{\notin K}(\theta) \leq t \leq n - l - I_{\notin K}$, the inequality
\begin{equation*}
	\left|\hat{F}_{q,n}(\theta)\right|  \leq \alpha_{\notin K} \left( n - l  - I_{\notin K} - t \right) q^t \left|\hat{F}_{q,{l}}(T^x\theta) \right|
\end{equation*}
holds.
\end{lemma}
\begin{proof}
Since every index $i \in \overline{K}$ is either in $Z_{\notin K}(\theta)$, $N_{\notin K}(\theta)$ or $\mathcal{I} \cap \overline{K}$, it is easy to see that
\begin{equation}\label{geq_equality_relation}
	\#Z_{\notin K}(\theta) + \#N_{\notin K} (\theta) + I_{\notin K} = n - l~.
\end{equation}
In particular, we have
\begin{equation}\label{n_from_z}
	\#N_{\notin K} (\theta) = n - l - I_{\notin K} - \#Z_{\notin K}(\theta)~.
\end{equation}
Inserting this into \eqref{bound_fhat_by_h_main_eq}, we have
\begin{equation}\label{fbound_explicit_n}
	\left|\hat{F}_{q,n}(\theta)\right|  \leq \alpha_{\notin K} \left( n - l  - I_{\notin K} - \# Z_{\notin K}(\theta) \right) q^{\# Z_{\notin K}(\theta)} \left|\hat{F}_{q,{l}}(T^x\theta) \right|~.
\end{equation}
Note that by definition we have $\alpha_{\notin K}(v + u) \leq \alpha_{\notin K}(v)q^u$ for every $v,u \geq 0$. Since by assumption $t$ satisfies  $\# Z_{\notin K}(\theta) \leq t \leq n - l - I_{\notin K}$, using this monotonicity argument with $u = t - \# Z_{\notin K}(\theta)$ and $v = n - l - I_{\notin K} - t$ yields
\begin{equation*}
	\alpha_{\notin K} \left( n - l  - I_{\notin K} - \# Z_{\notin K}(\theta) \right) \leq \alpha_{\notin K} \left( n - l  - I_{\notin K} - t \right)q^{t - \# Z_{\notin K}(\theta)}~.
\end{equation*}
Inserting this bound into \eqref{fbound_explicit_n} gives
\begin{equation*}
	\left|\hat{F}_{q,n}(\theta)\right| \leq 
	 \alpha_{\notin K} \left( n - l  - I_{\notin K} - t \right)q^t \left|\hat{F}_{q,{l}}(T^x\theta) \right|~,
\end{equation*}
as claimed.
\end{proof}


\begin{lemma}\label{sum_over_deg_bound_init_bigi}
Let $1 \leq h \leq n/2$, and $0 \leq x < n - 2h$. Define $K = K_{x, 2h} = \{x, \dots, x + 2h\}$, and $\overline{K}$ as in Lemma \ref{bound_fhat_by_h}. Write $I_{\in K} =  \#\left(\mathcal{I} \cap K \right)$, $I_{\notin K} = I - I_{\in K}$. Let $\alpha$ be defined as in \eqref{alpha_def}, and $Y_h$ be defined as in \eqref{y_h_def}. The inequality
\begin{equation*}
	Y_h
		\leq \alpha \left(2h - I_{\in K} \right)q^{n + 1 - I_{\notin K}}
\end{equation*}
holds. 
\end{lemma}
\begin{proof}
We define $Z_{\notin K}(\theta)$, $N_{\notin K}(\theta)$, and $\alpha_{\notin K}$ as in Lemma \ref{bound_fhat_by_h}. 
It is easy to see from \eqref{geq_equality_relation} that for every $\theta \in \mathcal{U}$
\begin{equation} \label{bound_using_i_bigi}
	\# Z_{\notin K}(\theta)  = n - 2h - I_{\notin K} - \# N_{\notin K}(\theta) \leq n - 2h - I_{\notin K}.
\end{equation}
Applying Lemma \ref{bound_fhat_by_z_bound} with $t = n - 2h - I_{\notin K}$, we get
\begin{eqnarray*}
	\left|\hat{F}_{q,n}(\theta)\right| &\leq& \alpha(0)q^{n - 2h - I_{\notin K}}\left|\hat{F}_{q,{2h}}(T^x\theta) \right| \\
						&=& q^{n - 2h - I_{\notin K}}\left|\hat{F}_{q,{2h}}(T^x\theta) \right| ~.
\end{eqnarray*}

For every pair $G$, $H$ with $H \neq 1,~T$ squarefree of degree $H$ and $G$ coprime to $H$ and of smaller degree, we associate $F_{G,H,x} \in \mathcal{M}_{2h}$ such that $\left|T^{-2h}F_{G,H,x} - T^xG/H\right| < q^{-2h}$. By Lemma~ \ref{diff_mult_t} we know that each $F \in \mathcal{M}_{2h}$ corresponds to at most $q$ such pairs.
Using this fact and Lemma~\ref{dont_care_fraction}, we obtain the inequality
\begin{equation*} 
	\sum_{\substack{G, H \\ \deg H = h}}|\hat{F}_{q,{2h}}(T^xG/H)| \leq 
			q\sum_{F \in \mathcal{M}_{2h}}\left| \hat{F}_{q,{2h}}\left(T^{-2h}F \right)\right| ~.
\end{equation*}
Finally, we have
\begin{eqnarray*}
	Y_h = \sum_{\substack{G, H \\ \deg H = h}} \left|\hat{F}_{q,n}(G/H)\right|
	&\leq& q^{n - 2h - I_{\notin K}}\sum_{\substack{G, H \\ \deg H = h}} \left|\hat{F}_{q,{2h}}(T^xG/H)\right| \\
	&\leq& q^{n - 2h - I_{\notin K}}q\sum_{F \in \mathcal{M}_{2h}}\left|\hat{F}_{q,{2h}}(T^{-2h}F)\right| ~,
\end{eqnarray*}
and applying Lemma \ref{gen_bound} on the sum gives us
\begin{eqnarray*}
	Y_h
		&\leq& q^{n + 1 - 2h - I_{\notin K}} \prod_{\substack{i \in \mathcal{J} \cap K}}(N_i+1) \cdot q^{2h} \\
	&\leq& \alpha \left(2h +  I_{\notin K} - I\right)q^{n + 1 - I_{\notin K}},
\end{eqnarray*}
as claimed.
\end{proof}


\begin{lemma}\label{sum_over_deg_bound_init}
Let $1 \leq h \leq n/2$, and choose $x = 0$. Define $K_{0, 2h}$ as in Lemma \ref{bound_fhat_by_h}. Write \mbox{$I_{\notin K} = \#\left(\mathcal{I} \backslash K_{0, 2h} \right)$}. Denote \mbox{$t_h = \max \{I_{\notin K},  \left\lfloor \frac{n}{h} \right\rfloor - 2\}$}, and let $\alpha$ and $Y_h$ be defined as before. The inequality
\begin{equation*}
	Y_h
		\leq \alpha \left(2h +  t_h - I\right)q^{n - t_h}
\end{equation*}
holds. 
\end{lemma}
\begin{proof}
We define $Z_{\notin K}$, $N_{\notin K}$, and $\alpha_{\notin K}$ as in Lemma \ref{bound_fhat_by_h}. 
Assume $\theta = G/H$, with \mbox{$H \notin \{1,T\}$} squarefree of degree $h$, and $G$ coprime to $H$ of degree $< h$. We give two different bounds for $ \# Z_{\notin K}(G/H)$. First, it is easy to see from \eqref{geq_equality_relation} that
\begin{equation} \label{bound_using_i}
	\# Z_{\notin K}(G/H)  = n - 2h - I_{\notin K} - \# N_{\notin K}(G/H) \leq n - 2h - I_{\notin K}.
\end{equation}
Second, by Lemma \ref{deg_zero_bound} we know that there are no $h$ consecutive zeros in $G/H$, so in particular there are at most $n - 2h - \left\lfloor \frac{n - 2h}{h} \right\rfloor$ zero coefficients between $-2h-1$ and $-n$. Hence
\begin{equation} \label{bound_using_lemma}
	\#  Z_{\notin K}(G/H) \leq n - 2h - \left\lfloor \frac{n - 2h}{h} \right\rfloor
					= n - 2h -  \left\lfloor \frac{n}{h} \right\rfloor + 2~.
\end{equation}
Considering the definition of $t_h$, \eqref{bound_using_i} and \eqref{bound_using_lemma} give
\begin{equation}\label{combined_bound}
	\# Z_{\notin K}(G/H) \leq n - 2h - t_h.
\end{equation}
Inserting this into Lemma \ref{bound_fhat_by_z_bound} yields
\begin{equation*}
	\left|\hat{F}_{q,n}(G/H)\right| \leq \alpha_{\notin K}\left(t_h - I_{\notin K}\right)q^{n - 2h - t_h}\left|\hat{F}_{q,{2h}}(G/H) \right|.
\end{equation*}

In a similar manner to Lemma~\ref{sum_over_deg_bound_init_bigi}, \mbox{$\left|T^{-2h}F_{G,H} - G/H\right| < q^{-2h}$} for some $F_{G,H} \in \mathcal{M}_{2h}$. By Lemma~ \ref{diff_mult_t} we know that each $F \in \mathcal{M}_{2h}$ corresponds to at most one pair $G$, $H$ as described.

We can now use Lemma~\ref{dont_care_fraction} to obtain
\begin{equation*}
	Y_h \leq \sum_{\substack{G, H \\ \deg H = h}}|\hat{F}_{q,{2h}}(G/H)| \leq \sum_{F \in \mathcal{M}_{2h}}|\hat{F}_{q,{2h}}(T^{-2h}F)| ~.
\end{equation*}
Finally, we have
\begin{eqnarray*}
	Y_h = \sum_{\substack{G, H \\ \deg H = h}} \left|\hat{F}_{q,n}(G/H)\right|
	&\leq& \alpha_{\notin K} \left(t_h - I_{\notin K} \right)  q^{n - 2h - t_h}\sum_{\substack{G, H \\ \deg H = h}} \left|\hat{F}_{q,{2h}}(T^xG/H)\right| \\
	&\leq& \alpha_{\notin K} \left(t_h - I_{\notin K} \right)  q^{n - 2h - t_h}\sum_{F \in \mathcal{M}_{2h}}\left|\hat{F}_{q,{2h}}(T^{-2h}F)\right| ~,
\end{eqnarray*}
and applying Lemma \ref{gen_bound} on the sum gives us
\begin{eqnarray*}
	Y_h
		&\leq& \alpha_{\notin K} \left( t_h  - I_{\notin K}\right)q^{n - 2h - t_h} \prod_{\substack{i \in \mathcal{J} \\ i < 2h}}(N_i+1) \cdot q^{2h} \\
	&\leq& \alpha \left(2h +  t_h - I\right)q^{n - t_h}~,
\end{eqnarray*}
as claimed.
\end{proof}

In Proposition \ref{sum_over_deg_bound}, we give our final bound on sums over fractions of the form $G/H$ with \mbox{$\deg H = h$}. Having established Lemma  \ref{sum_over_deg_bound_init_bigi} and Lemma \ref{sum_over_deg_bound_init}, most of the work is already accomplished. In order to establish Proposition \ref{sum_over_deg_bound}, we add most of the assumptions of Theorem \ref{main_thm}. We assume that  $N_i < q^\varepsilon$ for all $i \in \mathcal{J}$, with $\varepsilon < 1$.
\begin{proof}[Proof of Proposition \ref{sum_over_deg_bound}]
Recall that, by assumption, $\alpha(m) \leq q^{\varepsilon m}$ for all $m \geq 0$.
Let $1 \leq h \leq \min\{n/2,n/I\}$. Write $I_{\notin 2h} = \# (\mathcal{I} \cap \{2h, \dots, n-1\})$, and  \mbox{$t_h = \max \{I_{\notin 2h},  \left\lfloor \frac{n}{h} \right\rfloor - 2\}$}. Lemma \ref{sum_over_deg_bound_init} gives us that 
\begin{equation*}
		\sum_{\substack{G, H \\ \deg H = h}} \left|\hat{F}_{q,n}(G/H)\right| \leq \alpha \left(2h +  t_h - I\right)q^{n - t_h}~.
\end{equation*}
Thus from the assumption on $\alpha$ we get
\begin{eqnarray*}
	\sum_{\substack{G, H \\ \deg H = h}} \left|\hat{F}_{q,n}(G/H)\right| &\leq& q^{n - t_h + \varepsilon\left(2h +  t_h - I\right)} \\
		&=& 	q^{n - t_h(1 - \varepsilon) + \varepsilon\left(2h - I \right)}~.
\end{eqnarray*}
Since $t_h \geq \lfloor n / h \rfloor - 2$ and $\varepsilon < 1$, we get
\begin{eqnarray*}
	\sum_{\substack{G,H \\ \deg H = h}} \left|\hat{F}_{q,n}(G/H)\right|
		&\leq& q^{n - (\left\lfloor n/h \right\rfloor - 2)(1 - \varepsilon) + \varepsilon\left( 2h - I\right)} \\
		&\leq& q^{n - (n/h - 3)(1 - \varepsilon) + \varepsilon\left( 2h - I\right)} \\
		&=&  q^{n + 3(1 - \varepsilon) - (1 - \varepsilon)n/h + 2\varepsilon h - \varepsilon I}~,
\end{eqnarray*}
which is the first part of the proposition.

We now move to the case where $n/I \leq h \leq \min\{n/2,n/I\}$. For every choice of $0 \leq x \leq n - 2h$, write $K = K_{x,l}$ and $I_{\notin K} = \# (\mathcal{I} \backslash K)$. Lemma \ref{sum_over_deg_bound_init_bigi} gives us that 
\begin{equation*}
		Y_h \leq \alpha \left(2h +  I_{\notin K} - I\right)q^{n + 1 - I_{\notin K}}~.
\end{equation*}
Bounding $\alpha \left(2h +  I_{\notin K} - I\right)$ by $q^{\varepsilon \left(2h +  I_{\notin K} - I\right)}$ yields
\begin{equation*}
		Y_h \leq q^{n + 1 - I_{\notin K} + \varepsilon \left(2h +  I_{\notin K} - I\right)}~,
\end{equation*}
and simplifying yields
\begin{equation*}
		Y_h\leq q^{n - I + 1}q^{(1 - \varepsilon)(I - I_{\notin K}) + 2\varepsilon h}~.
\end{equation*}
Choosing the optimal $x$ in the sense of Lemma \ref{keta_exists}, we can assume $I_{\notin K} \geq I - 4hI/n$ in the case where $h \leq n/2$ and $I_{\notin K} \geq I - 3hI/n$ when $h < n/4$, thus
\begin{equation*}
		Y_h \leq 
		q^{n - I + 1}q^{(1 - \varepsilon)4Ih/n + 2\varepsilon h}
		= q^{n - I + 1}q^{\left(2\varepsilon + (1 - \varepsilon)4I/n \right) h}
\end{equation*}
for all $h \leq n/2$, and similarly
\begin{equation*}
		Y_h \leq
		q^{n - I + 1}q^{\left(2\varepsilon + (1 - \varepsilon)3I/n \right) h}
\end{equation*}
when $h < n/4$. 
This completes the proof of the proposition.
\end{proof}

\section{Proof of the Theorem \ref{main_thm}}\label{full_proof}

\subsection{Auxiliary Results}\label{fullproof_aux}
The following lemmas apply Proposition \ref{sum_over_deg_bound} to the setting of \mbox{Theorem \ref{main_thm}}. In all the lemmas in this section we assume the setting of Proposition \ref{sum_over_deg_bound}.  We have results of two types - when $I < o\left(n/\log(n)\right)$, we get a strong bound on the error term. When $I$ is larger, we take more care to derive the most from our methods.

\begin{lemma}\label{degpart_x1}
Assume that $\varepsilon < 1/2$. Denote $l = \min\{n/I, n/2\}$, $s = \sqrt{(1 - \varepsilon)/(1 - 2\varepsilon)}\sqrt{n}$, $m_{n,I,\varepsilon} = \min \{n/I,~s\}$, and  $y = \frac{I \cdot s}{n}$. Define $Y_h$ as in \eqref{y_h_def}. The inequality
\begin{equation}
	\sum_{h = 1}^{l} q^{-h} Y_h \leq 
		nq^{n  -I}q^{3(1 - \varepsilon) - (1 - 2\varepsilon)m_{n,I,\varepsilon}}
\end{equation}
holds.
If in addition $y > 1$, we have
\begin{equation}
	\sum_{h = 1}^{n/I} q^{-h} Y_h \leq 
		C_{q, \varepsilon, y}q^{n  -I}q^{3(1 - \varepsilon) - (1 - 2\varepsilon)n/I}~,
\end{equation}
with $C_{q, \varepsilon, y}$ tending to $1$ as $q$ tends to infinity.
\end{lemma}
\begin{proof}
By \eqref{prop_frac_1}, we have
\begin{equation*}
	X := \sum_{h = 1}^{l} q^{-h} Y_h \leq 
	\sum_{h = 1}^{l} q^{-h}q^{n + 3(1 - \varepsilon) - (1 - \varepsilon)n/h + 2\varepsilon h - \varepsilon I}~.
\end{equation*}
By simplifying we get
\begin{equation}\label{x_by_rnh}
	X \leq q^{n + 3(1 - \varepsilon) - \varepsilon I} \sum_{h = 1}^{l} q^{- (1 - 2\varepsilon)h - (1 - \varepsilon)n/h}~.
\end{equation}
For ease of exposition, we denote the term $- (1 - 2\varepsilon)h - (1 - \varepsilon)n/h$ by $r_{n,\varepsilon}(h)$.
By \eqref{x_by_rnh}, in order to bound~$X$ it suffices to bound $ \sum_{h = 1}^{l}q^{r_{n,\varepsilon}(h)}$.

By deriving $r_{n, \varepsilon}$, we obtain that its maximum is attained at \mbox{$h_{\max} = s = \sqrt{(1 - \varepsilon)/(1 - 2\varepsilon)}\sqrt{n}$}. 
In the case where $y = \frac{I\cdot s}{n} \leq 1$, we use the union bound and get
\begin{eqnarray*}
	\sum_{h = 1}^{l}q^{r_{n,\varepsilon}(h)} &\leq& lq^{r_{n,\varepsilon}(s)} \leq nq^{r_{n,\varepsilon}(s)} \nonumber \\
		&=& nq^{-(1 - 2\varepsilon)s}q^{-(1 - \varepsilon)n/s}~.
\end{eqnarray*}
Noting that in this case $n/s \geq I$, we get
\begin{equation}\label{y_small_x1_case}
	\sum_{h = 1}^{l}q^{r_{n,\varepsilon}(h)} \leq nq^{-(1 - 2\varepsilon)s}q^{-(1 - \varepsilon)I}~.
\end{equation}

We now bound $ \sum_{h = 1}^{l}q^{r_{n,\varepsilon}(h)}$ when $y > 1$. For all $h < n/I$ we have 
\begin{equation}
	\frac{n}{h-1} - \frac{n}{h} = \frac{n}{h (h-1)} > \frac{I^2}{n} = \frac{y^2n}{s^2} = \frac{1 - 2\varepsilon}{1 - \varepsilon}y^2~,
\end{equation}
thus
\begin{eqnarray*}
	r_{n,\varepsilon}(h) - r_{n,\varepsilon}(h-1) &=& -1 + 2\varepsilon + (1 - \varepsilon)\left(\frac{n}{h-1} - \frac{n}{h}\right) \\
		&\geq& -1 + 2\varepsilon + (1 - \varepsilon) \frac{1 - 2\varepsilon}{1 - \varepsilon}y^2 \\
		&=& (1 - 2\varepsilon)(y^2 - 1)~.
\end{eqnarray*}
Write $m =  \lfloor n/I \rfloor$. 
By induction we obtain that for every $0 \leq j < m$ we have 
\begin{eqnarray*}
	r_{n,\varepsilon}(m - j) &\leq&  r_{n,\varepsilon}(m) - j(1 - 2\varepsilon)(y^2 - 1)~. \\
		&\leq& r_{n,\varepsilon}(n/I) - j(1 - 2\varepsilon)(y^2 - 1)~,
\end{eqnarray*}
where the latter inequality is due to monotonicity of $r_{n,\varepsilon}$ in the range $1 \leq h \leq n/I$ when $y > 1$. This means that we can bound $\sum_{h = 1}^{l}q^{r_{n,\varepsilon}(h)}$ by a geometric sum:
\begin{eqnarray}\label{y_large_x1_case}
	\sum_{h = 1}^{l}q^{r_{n,\varepsilon}(h)} &=& \sum_{j = 0}^{m - 1} q^{r_{n,\varepsilon}(n/I) - j(1 - 2\varepsilon)(y^2 - 1)} \nonumber \\
	&\leq& q^{r_{n,\varepsilon}(n/I)}
			\sum_{j = 0}^{m - 1}q^{-j(1-2\varepsilon)(y^2-1)} \nonumber \\
	&=& q^{-(1-2\varepsilon)n/I}q^{-(1-\varepsilon)I} \sum_{j = 0}^{m - 1}q^{-j(1-2\varepsilon)(y^2-1)}~.
\end{eqnarray}
We now turn to bound the right-hand sum $S = \sum_{j = 0}^{m - 1}q^{-j(1-2\varepsilon)(y^2-1)}$ in two ways. First, we note that $(1-2\varepsilon)(y^2-1) > 0$, hence a simple union bound gives us
\begin{equation*}
	S = \sum_{j = 0}^{m - 1}q^{-j(1-2\varepsilon)(y^2-1)} \leq m-1 < n~,
\end{equation*}
which combined with \eqref{y_large_x1_case} gives
\begin{equation}\label{union_sumj_bound}
	\sum_{h = 1}^{l}q^{r_{n,\varepsilon}(h)} < nq^{-(1-2\varepsilon)n/I}q^{-(1-\varepsilon)I} ~.
\end{equation}
Second, we treat $S$ as a geometric series, in which case we bound it by the infinite series
\begin{equation}\label{geom_sumj_bound}
	S = \sum_{j = 0}^{m - 1}q^{-j(1-2\varepsilon)(y^2-1)} \leq \frac{1}{1 - q^{-(1-2\varepsilon)(y^2-1)}}~.
\end{equation}
Writing 
\begin{equation*}
C_{q, \varepsilon, y} = \left(1 - q^{-(1-2\varepsilon)(y^2-1)}\right)^{-1}~,
\end{equation*} we note that $C_{q, \varepsilon, y}$ tends to $1$ as $q$ tends to infinity.
Inserting \eqref{geom_sumj_bound} into \eqref{y_large_x1_case} then gives
\begin{equation}\label{geom_sumj_bound2}
	\sum_{h = 1}^{n/I}q^{r_{n,\varepsilon}(h)} \leq C_{q, \varepsilon, y}q^{-(1-2\varepsilon)n/I}q^{-(1-\varepsilon)I}~,
\end{equation}
and using this in  \eqref{x_by_rnh} yields that when $y > 1$,
\begin{equation}
	X \leq C_{q, \varepsilon, y}q^{n  -I}q^{3(1 - \varepsilon) - (1 - 2\varepsilon)m_{n,I,\varepsilon}}~,
\end{equation}
with $C_{q, \varepsilon, y}$ tending to $1$ as $q$ tends to infinity. This concludes the second part of the lemma.

For the first part, recall that $m_{n,I,\varepsilon} = \min \{n/I,s\}$. Inserting \eqref{y_small_x1_case} in the case where $y \leq 1$ and \eqref{union_sumj_bound} in the case where $y > 1$ into \eqref{x_by_rnh} gives
\begin{equation}
	X = q^{n  + 3(1 - \varepsilon) - \varepsilon I}\sum_{h = 1}^{l}q^{r_{n,\varepsilon}(h)} \leq nq^{n  -I}q^{3(1 - \varepsilon) - (1 - 2\varepsilon)m_{n,I,\varepsilon}}
\end{equation}
for all $y$.
This completes the proof of the lemma.
\end{proof}

\begin{lemma}\label{degpart_big_bound} Let $n/I \leq k \leq n/2$. Assuming $\frac{I}{n} < \frac{1}{4} \cdot \frac{1-2\varepsilon}{1-\varepsilon}(1 - \tau)$ for some $\tau > 0$, we have 
\begin{equation}
	\sum_{h = k}^{n/2} q^{-h} Y_h < 
		C_{q,\varepsilon,\tau}q^{n-I+1}q^{-k\left(1 - 2\varepsilon - 4(1 - \varepsilon)I/n\right)}~,
\end{equation}
with $C_{q,\varepsilon,\tau}$ tending to $1$ as $q$ tends to infinity.
\end{lemma}
\begin{proof}
By the \eqref{prop_frac_2}, we know that 
$Y_h \leq q^{n -I + 1}q^{\left(2\varepsilon + (1 - \varepsilon)4I/n \right) h}$ for all $k \leq h \leq n/2$. 
Thus
\begin{eqnarray*}
	\sum_{h = k}^{n/2} q^{-h}Y_h
			&\leq& \sum_{h = k}^{n/2} q^{-h}q^{n -I + 1}q^{\left(2\varepsilon + (1 - \varepsilon)4I/n \right) h}\\
		&=& q^{n-I+1}\sum_{h = k}^{n/2}q^{-h\left(1 - 2\varepsilon - 4(1 - \varepsilon)I/n \right)}
\end{eqnarray*}
Substituting $r = q^{-\left(1 - 2\varepsilon - 4(1 - \varepsilon)I/n \right)}$, we get
\begin{equation*}
		\sum_{h = k}^{n/2} q^{-h}Y_h \leq q^{n - I + 1} \sum_{h = k}^{n/2}r^h~.
\end{equation*}
Note that from the assumption, $1 - 2\varepsilon - 4(1 - \varepsilon)I/n > (1 - 2\varepsilon)\tau > 0$. So 
\begin{equation}\label{r_bound}
r < q^{-(1 - 2\varepsilon)\tau} < 1~,
\end{equation}
and the sum is geometric. Thus we can bound it as
\begin{eqnarray*}
		\sum_{h = k}^{n/2} q^{-h} \sum_{\substack{G,H \\ \deg H = h}} \left|\hat{F}_{q,n}(G/H)\right|
			&\leq& q^{n - I + 1} \cdot r^{k} \cdot \frac{1}{1-r}\\
		&\leq& \frac{1}{1-r}q^{n-I+1}q^{-k\left(1 - 2\varepsilon - 4(1 - \varepsilon)I/n\right)}~.
\end{eqnarray*}
Write
\begin{equation*}
 C_{q,\varepsilon,\tau} =  \left(1 - q^{-(1-2\varepsilon)\tau}\right)^{-1}~.
\end{equation*}
Then $C_{q,\varepsilon,\tau}$ tends to $1$ as $q$ tends to infinity, and from \eqref{r_bound} we get that $1/(1-r) < C_{q,\varepsilon,\tau}$. Thus
\begin{equation*}
		\sum_{h = n/4}^{n/2} q^{-h} \sum_{\substack{G,H \\ \deg H = h}} \left|\hat{F}_{q,n}(G/H)\right|
			< C_{q,\varepsilon,\tau}q^{n-I+1}q^{-k\left(1 - 2\varepsilon - 4(1 - \varepsilon)I/n\right)}~,
\end{equation*}
as needed.
\end{proof}

\begin{lemma}\label{degpart_small_bound} Assuming $\frac{I}{n} < \frac{1}{4} \cdot \frac{1-2\varepsilon}{1-\varepsilon}(1 - \tau)$, we have 
\begin{equation}
	\sum_{n/I \leq h < n/4} q^{-h} \sum_{\substack{G,H \\ \deg H = h}} \left|\hat{F}_{q,n}(G/H)\right| < 
		C_{q,\varepsilon,\tau}q^{n-I+1}q^{-(1 - 2\varepsilon - 3(1 - \varepsilon)I/n)n/I}~,
\end{equation}
where $C_{q,\varepsilon,\tau}$ is given in Lemma \ref{degpart_big_bound}.
\end{lemma}
\begin{proof}

The proof is essentially the same as that of Lemma \ref{degpart_big_bound}, the difference being that we use \eqref{prop_frac_3} instead of \eqref{prop_frac_2}.
Writing 
\mbox{$r_2 = q^{-\left(1 - 2\varepsilon - 3(1 - \varepsilon)I/n \right)}$}, we have $r_2 < q^{-(1 - 2\varepsilon)\tau} < 1$, and
\begin{eqnarray*}
		\sum_{h = n/I}^{n/4} q^{-h} \sum_{\substack{G,H \\ \deg H = h}} \left|\hat{F}_{q,n}(G/H)\right|
			&\leq& \frac{1}{1-r_2}q^{n-I+1}r_2^{n/I} \\
			&\leq& C_{q,\varepsilon,\tau}q^{n-I+1}q^{-(1 - 2\varepsilon - 3(1 - \varepsilon)I/n)n/I}~,
\end{eqnarray*}
as required.
\end{proof}

\begin{lemma}\label{err_term_smalli_bnd}
	Assume that $\frac{I}{n} < \frac{1}{4} \cdot \frac{1-2\varepsilon}{1-\varepsilon}(1 - \tau)$. Denote 
	$s = \sqrt{(1 - \varepsilon)/(1 - 2\varepsilon)}\sqrt{n}$, $y = \frac{I \cdot s}{n}$, and \mbox{$m_{n,I,\varepsilon} = \min \{n/I,~s\}$}. Define $Y_h$ as in \eqref{y_h_def}.
	If $I = o\left(n/\log(n)\right)$, then
	\begin{equation*}
		\sum_{h = 1}^{n/2} q^{-h}Y_h \leq
			q^{n-I}q^{-\left(1 - 2\varepsilon + o(1)\right)m_{n,I,\varepsilon}}~.
	\end{equation*}
\end{lemma}
\begin{proof}
We begin by partitioning the sum $	\sum_{h = 1}^{n/2} q^{-h}Y_h$ into two parts:
\begin{equation*}
	\sum_{h = 1}^{n/2} q^{-h}D_h =\quad \underbrace{\sum_{h = 1}^{\min\{n/2, n/I\}} q^{-h}Y_h}_{=:X_1} \quad + \quad \underbrace{\sum_{h = n/I}^{n/2} q^{-h}Y_h}_{=:X_2}~.
\end{equation*}
Since $I = o\left(n/\log(n)\right)$, it follows that $\log_q(n) = o(n/I)$. 
Applying Lemma \ref{degpart_x1} gives
\begin{eqnarray*}
	X_1 &\leq& \frac{n}{2}q^{n  -I}q^{3(1 - \varepsilon) - (1 - 2\varepsilon)m_{n,I,\varepsilon}}  \\
		&\leq& q^{n-I}q^{-(1 - 2\varepsilon)m_{n,I,\varepsilon} + \log_q(n) + 3(1 - \varepsilon)}~.
\end{eqnarray*}
Since $\log_q(n) = o(m_{n,I,\varepsilon})$ we obtain
\begin{equation}
	X_1 \leq q^{n-I}q^{-\left(1 - 2\varepsilon + o(1)\right)m_{n,I,\varepsilon}}~.
\end{equation} 
Using Lemma \ref{degpart_big_bound} with $k = n/I$ gives
\begin{eqnarray*}
	X_2 &\leq& C_{q,\varepsilon,\tau}q^{n-I+1}q^{-\left(1 - 2\varepsilon + 4(1 - \varepsilon)I/n\right)n/I} \\
		&=& q^{n-I}q^{-(1 - 2\varepsilon)n/I + 1 + 4(1 - \varepsilon) + \log_q(C_{q,\varepsilon,\tau})} \\
		&=& q^{n-I}q^{-(1-2\varepsilon + o(1))n/I}~,
\end{eqnarray*}
and since $n/I \leq m_{n,I,\varepsilon}$ we get
\begin{eqnarray*}
	\sum_{h = 1}^{n/2} q^{-h}Y_h &=& X_1 + X_2 ~\leq~ 2q^{n-I}q^{-\left(1 - 2\varepsilon + o(1)\right)m_{n,I,\varepsilon}} \\
	&=& q^{n-I}q^{-\left(1 - 2\varepsilon + o(1)\right)m_{n,I,\varepsilon}}~.
\end{eqnarray*}
This concludes the proof of the lemma.

\end{proof}

\begin{lemma}\label{err_term_bigi_bnd}
	Assume that $\frac{I}{n} < \frac{1}{4} \cdot \frac{1-2\varepsilon}{1-\varepsilon}(1 - \tau)$. Denote 
	$s = \sqrt{(1 - \varepsilon)/(1 - 2\varepsilon)}\sqrt{n}$, $y = \frac{I \cdot s}{n}$. Define $Y_h$ as in \eqref{y_h_def}.
	Assume that $y > 1$ and $n$ is sufficiently large in terms of $\varepsilon$ and $\tau$. The inequality
	\begin{equation*}
		\sum_{h = 1}^{n/2} q^{-h}Y_h \leq
			q^{n-I}q^{-(1 - 2\varepsilon)n/I + 4 - 3 \varepsilon + B_{q,\varepsilon,\tau,y}}~
	\end{equation*}
	holds, with $B_{q,\varepsilon,\tau,y}$ tending to zero as $q$ grows to infinity.
\end{lemma}
\begin{proof}
	We give a partition
	to $\sum_{h = 1}^{n/2} q^{-h}Y_h$:
\begin{equation}\label{partition_of_three}
	\sum_{h = 1}^{n/2} q^{-h}Y_h = \quad \underbrace{\sum_{h = 1}^{\min{n/2, n/I}} q^{-h}Y_h}_{=:X_1} \quad
	+ \quad \underbrace{\sum_{n/I \leq h < n/4} q^{-h}Y_h}_{=:X_2} \quad
	+ \quad \underbrace{\sum_{h = n/4}^{n/2} q^{-h}Y_h}_{=:X_3}
\end{equation}
Under the assumptions of the lemma we know that $I > \sqrt{(1-2\varepsilon)/(1-\varepsilon)}\sqrt{n}$, hence by considering $n$ that are sufficiently large in terms of $\varepsilon$ we may assume that $n/I < n/4$.
We bound $X_1$ using the second part of Lemma  \ref{degpart_x1}:
\begin{equation*}
	X_1 \leq C_{q, \varepsilon, y}q^{n  -I}q^{3(1 - \varepsilon) - (1 - 2\varepsilon)n/I}~,
\end{equation*}
with $C_{q, \varepsilon, y}$ tending to $1$ as $q$ tends to infinity.
We use Lemma \ref{degpart_small_bound} in order to bound $X_2$, and that gives
\begin{equation*}
	X_2 	\leq C_{q,\varepsilon,\tau}q^{n-I+1}q^{-\left(1 - 2\varepsilon - 3(1 - \varepsilon)I/n\right)n/I}~,
\end{equation*}
with $C_{q,\varepsilon,\tau}$ also tending to $1$ as $q$ tends to infinity.
For $X_3$, we use Lemma \ref{degpart_big_bound} with $k = n/4$. From that we have
\begin{eqnarray*}
	X_3 &\leq& q^{n-I+1}q^{-\left(1 - 2\varepsilon - 4(1 - \varepsilon)I/n\right)n/4}~.
\end{eqnarray*}
Now we show that if $n$ is sufficiently large we can guarantee that the bound on $X_2$ dominates the bound on $X_3$. 
This happens when
\begin{equation}\label{no_beta_yet}
	-\left(1 - 2\varepsilon - 3(1 - \varepsilon)I/n\right)n/I \geq -\left(1 - 2\varepsilon - 4(1 - \varepsilon)I/n\right)n/4~.
\end{equation}
Writing $\beta = 1 - 2\varepsilon - 4(1 - \varepsilon)I/n$, this inequality translates into
\begin{equation}
	-\left(\beta + (1-\varepsilon)I/n\right) n/I \geq -\beta n/4~.
\end{equation}
Simplifying this inequality gives
\begin{align}\label{n_beta_ineq}
	-\beta n/I - 1 - \varepsilon ~&\geq~ - \beta n / 4 \nonumber\\
	n/4 - n/I ~& \geq~ (1 - \varepsilon)/\beta 
\end{align}
Note that the assumption $\frac{I}{n} < \frac{1}{4} \cdot \frac{1-2\varepsilon}{1-\varepsilon}(1 - \tau)$ gives us that
$\beta = 1 - 2\varepsilon - 4(1 - \varepsilon)I/n \geq \tau$, hence $(1 - \varepsilon)/\beta <  (1 - \varepsilon)/\tau$. Since $I \gg_\varepsilon \sqrt{n}$, the left-hand side of \eqref{n_beta_ineq} behaves asymptotically like $n$, so for large $n$ we get
\begin{equation*}
	n/4 - n/I ~\geq~ (1 - \varepsilon)/\tau ~\geq~ (1 - \varepsilon)/\beta~.
\end{equation*}
This implies that
\begin{eqnarray*}
	X_2, X_3 &\leq& C_{q,\varepsilon,\tau}q^{n-I+1}q^{-\left(1 - 2\varepsilon - 3(1 - \varepsilon)I/n\right)n/I} \\
		&=& C_{q,\varepsilon,\tau}q^{n-I}q^{-(1-2\varepsilon)n/I + 4 - 3\varepsilon}~.
\end{eqnarray*}
Thus
\begin{eqnarray*}
	X_1 + X_2 + X_3 &\leq& C_{q, \varepsilon, y}q^{n  -I}q^{3(1 - \varepsilon) - (1 - 2\varepsilon)n/I} +
			2C_{q,\varepsilon,\tau}q^{n-I}q^{-(1-2\varepsilon)n/I + 4 - 3\varepsilon} \\
			&\leq& q^{n-I}q^{-(1-2\varepsilon)n/I + 4 - 3\varepsilon + \log_q(C_{q, \varepsilon, y} + 2C_{q,\varepsilon,\tau})}~.
\end{eqnarray*}
Writing $B_{q,\varepsilon,\tau,y} = \log_q(C_{q, \varepsilon, y} + 2C_{q,\varepsilon,\tau})$ and recalling that $C_{q,\varepsilon, y} $ and $C_{q, \varepsilon, \tau}$ are bounded with respect to $q$, it is clear that $B_{q,\varepsilon,\tau,y}$ tends to zero as $q$ tends to infinity. 
Inserting this into \eqref{partition_of_three} gives
\begin{equation*}
		\sum_{h = 1}^{n/2} q^{-h}\sum_{\substack{G,H \\ \deg H = h}} \left|\hat{F}_{q,n}(G/H)\right| \leq
			q^{n-I}q^{-(1 - 2\varepsilon)n/I + 4 - 3 \varepsilon + B_{q,\varepsilon,\tau,y}}~,
	\end{equation*} 
as required.

\end{proof}

\subsection{A partition}

Recall that as in \eqref{parseval_formula}
\begin{equation}\label{fourier_transform_main}
	\sum_{P \in \mathcal{C}} 1 = \frac{1}{q^n}\sum_{F \in \mathcal{M}_n}\hat{F}_{q,n}(T^{-n}F)f(-T^{-n}F),
\end{equation}
with $f(\theta) = \sum_{P \in \mathcal{M}_n}\bold{e}(\theta P)$.
For every polynomial $F \in \mathcal{M}_n$ we denote $\theta_F = \{T^{-n}F\}$. We denote by $G_F,H_F$ the corresponding polynomials to $\theta_F$ as in Lemma \ref{dirichlet}. We divide $\mathcal{M}_n$ into three sets, with relation to Lemma \ref{circ_method_ext}:
\begin{align}
	\mathcal{S}_1 &= \left\{ F  \in \mathcal{M}_n : \theta_F = c/T \mbox{ for some } c \in \mathbb{F}_q \right\}, \nonumber \\
	\mathcal{S}_2 &= \left\{ F \in \mathcal{M}_n : |\theta_F - G_F/H_F| < q^{-n} \mbox{ and } 1,T \neq H_F \mbox{ is squarefree} \right\}, \label{m_n_partition}\\
	\mathcal{S}_3 &= M_n \backslash \left( \mathcal{S}_1 \cup \mathcal{S}_2 \right)~. \nonumber
\end{align} 
Note that this is indeed a partition of $\mathcal{M}_n$. The sum \eqref{fourier_transform_main} decomposes into  three sums accordingly. The sum over the polynomials in $\mathcal{S}_1$ will give us the main term, which we compute in Subsection \ref{main_term_sub}. We use the bounds obtained in Section \ref{fhat_section} in order to show that the sums over $\mathcal{S}_2,\mathcal{S}_3$ are of small size in Subsection \ref{error_term_sub}. The conclusion of the proof is given in Section \ref{conclusion_sub}.

\subsection{Error Term Bound}\label{error_term_sub}
Our aim is to bound
\begin{equation} \label{error_term_y}
	Y = \left| \frac{1}{q^n}\sum_{F \in \mathcal{S}_3}\hat{F}_{q,n}(\theta_F)f(-\theta_F) + \frac{1}{q^n}\sum_{F \in \mathcal{S}_2}\hat{F}_{q,n}(\theta_F)f(-\theta_F) \right|.
\end{equation}
First, one can easily check by the definition of $f$ in \eqref{exp_sum_defn} that for every $\theta \in \mathcal{U}$ one has \mbox{$|f(-\theta)| = |f(\theta)|$}. We now apply circle-method bounds on $f$ given in Lemma \ref{circ_method_ext}. For $F \in  \mathcal{S}_2$, we have 
\begin{equation}\label{fhat_s2_bound}
	|f(-\theta_F)| =  |f(\theta_F)| \leq q^{n - \frac{1}{2}\lfloor \frac{n}{2} \rfloor} + q^{n - \deg H_F}~. 
\end{equation}
For $F \in \mathcal{S}_3$ we have 
\begin{equation}\label{fhat_s3_bound}
	|f(-\theta_F)| =  |f(\theta_F)| \leq q^{n -  \frac{1}{2}\lfloor \frac{n}{2} \rfloor}~.
\end{equation} 
Applying the triangle inequality to \eqref{error_term_y} yields
\begin{equation*}
	Y \leq \frac{1}{q^n}\sum_{F \in \mathcal{S}_3}\left|\hat{F}_{q,n}(\theta_F)\right| \cdot \left| f(-\theta_F) \right| +  \frac{1}{q^n}\sum_{F \in \mathcal{S}_2}\left|\hat{F}_{q,n}(\theta_F)\right| \cdot \left| f(-\theta_F) \right|~.
\end{equation*}
By \eqref{fhat_s2_bound} and \eqref{fhat_s3_bound} we obtain
\begin{eqnarray*}
	Y &\leq&  \frac{1}{q^n} \sum_{F \in \mathcal{S}_3}\left| \hat{F}_{q,n}(\theta_F)\right| q^{n - \frac{1}{2}\lfloor \frac{n}{2}\rfloor} + \frac{1}{q^n} \sum_{F \in \mathcal{S}_2} \left| \hat{F}_{q,n}(\theta_F)\right| \left(q^{n - \frac{1}{2}\lfloor \frac{n}{2} \rfloor} +  q^{n - \deg H_F}\right)~.
\end{eqnarray*}
Simplifying this gives
\begin{eqnarray*}
	Y &\leq&  \frac{q^{n - \frac{1}{2}\lfloor \frac{n}{2}\rfloor}}{q^n} \sum_{F \in \mathcal{S}_2 \cup \mathcal{S}_3}\left| \hat{F}_{q,n}(\theta_F)\right| + \frac{1}{q^n} \sum_{F \in \mathcal{S}_2} q^{n - \deg H_F} \left| \hat{F}_{q,n}(\theta_F)\right| \\
	& \leq & \frac{q^{n - \frac{1}{2}\lfloor \frac{n}{2} \rfloor}}{q^n}\sum_{F \in \mathcal{M}_n} \left|\hat{F}_{q,n}(\theta_F)\right|
	 + \frac{1}{q^n} \sum_{G,H}
		  q^{n - \deg H_F}  \left| \hat{F}_{q,n}(G/H)\right|, \end{eqnarray*}
where in the latter sum $H$ ranges over squarefree polynomials other than $1, T$, and $G$ is coprime to $H$ of smaller degree. Note that we replace $ \left|\hat{F}_{q,n}(\theta_F)\right|$ by $\left| \hat{F}_{q,n}(G/H)\right|$ in the right-hand sum due to Lemma~\ref{dont_care_fraction}. Applying Proposition \ref{gen_bound} to the first sum and changing order of summation in the second yields
\begin{equation}\label{alpha_plus_aux}
	Y  \leq \alpha(n-I)q^{n - \frac{1}{2}\lfloor \frac{n}{2} \rfloor} + \frac{1}{q^n} \sum_{h = 1}^{n/2} ~q^{n-h}\sum_{\substack{G,H \\ \deg H = h}} \left|\hat{F}_{q,n}(G/H)\right|~,
\end{equation}
where again $H$ ranges over squarefree polynomials other than $1, T$. 
We assume that $N_i < q^{\varepsilon m}$ for all $0 \leq i < n$, 
and that $\frac{I}{n} < \frac{1}{4} \cdot \frac{1-2\varepsilon}{1-\varepsilon}(1 - \tau)$ for some $\tau > 0$. Denote 
$s = \sqrt{(1 - \varepsilon)/(1 - 2\varepsilon)}\sqrt{n}$, and $m_{n,I,\varepsilon} = \min\{n/I, s\}$. By Lemma \ref{err_term_smalli_bnd}, if $I = o\left(n/\log(n)\right)$ then
\begin{equation*}
\sum_{h = 1}^{n/2} q^{-h}\sum_{\substack{G,H \\ \deg H = h}} \left|\hat{F}_{q,n}(G/H)\right| \leq
	q^{n-I}q^{-\left(1 - 2\varepsilon + o(1)\right)m_{n,I,\varepsilon}}~.
\end{equation*}
In this case by \eqref{alpha_plus_aux} we obtain a bound on the error term:
\begin{equation} \label{eq_error_bound_smalli}
 \left| \left(\sum_{P \in \mathcal{C}} 1\right)  - \frac{1}{q^n}\sum_{F \in \mathcal{S}_1}\hat{F}_{q,n}(\theta_F)f(-\theta_F) \right| 
 	\leq \alpha(n-I)q^{n - \frac{1}{2}\lfloor \frac{n}{2} \rfloor} + q^{n-I}q^{-\left(1 - 2\varepsilon + o(1)\right)m_{n,I,\varepsilon}}~.
 \end{equation}

By Lemma \ref{err_term_bigi_bnd}, if we keep our notation and assume that $y = \frac{I \cdot s}{n} > 1$ and that $n$ is sufficiently large in terms of $\varepsilon$ and $\tau$, we obtain
\begin{equation*}
		\sum_{h = 1}^{n/2} q^{-h}\sum_{\substack{G,H \\ \deg H = h}} \left|\hat{F}_{q,n}(G/H)\right| \leq
			q^{n-I}q^{-(1 - 2\varepsilon)n/I + 4 - 3 \varepsilon + B_{q,\varepsilon,\tau,y}}~,
\end{equation*}
with $B_{q,\varepsilon,\tau,y}$ tending to zero as $q$ grows to infinity. Substituting this bound into \eqref{alpha_plus_aux} gives an error term bound of
\begin{equation} \label{eq_error_bound_bigi}
 	\left| \left(\sum_{P \in \mathcal{C}} 1\right)  - \frac{1}{q^n}\sum_{F \in \mathcal{S}_1}\hat{F}_{q,n}(\theta_F)f(-\theta_F) \right| 
	\leq \alpha(n-I)q^{n - \frac{1}{2}\lfloor \frac{n}{2} \rfloor} + 
	q^{n-I}q^{-(1 - 2\varepsilon)n/I + 4 - 3 \varepsilon + B_{q,\varepsilon,\tau,y}}~.
\end{equation}

\subsection{Main Term Computation}\label{main_term_sub}
For the main term, we have
\begin{equation*}
X = \frac{1}{q^n}\sum_{F \in \mathcal{S}_1}\hat{F}_{q,n}(\theta_F)f(-\theta_F) = 
\frac{1}{q^n}\sum_{a \in \mathbb{F}_q} \hat{F}_{q,n}(a/T)f(-a/T)
\end{equation*}
Expanding out the definition of $f$ and $\hat{F}_{q,n}$ given in  \eqref{exp_sum_defn} and \eqref{fhat_def}, we get
\begin{eqnarray*} X &=& 
	\frac{1}{q^n} \sum_{a \in \mathbb{F}_q} \sum_{F \in \mathcal{M}_n} \mathds{1}_\mathcal{C}(F)\bold{e}\left(F \cdot \frac{a}{T}\right) \sum_{P \in \mathcal{M}_n} \bold{e}\left(-\frac{a}{T}P\right),
\end{eqnarray*}
and changing order of summation gives
\begin{eqnarray*}
	X &=& \frac{1}{q^n} \sum_{F, P \in \mathcal{M}_n} \mathds{1}_\mathcal{C}(F) \sum_{a \in \mathbb{F}_q} \bold{e}\left(\frac{(F - P)a}{T}\right) \\
	&=& \frac{1}{q^n} \sum_{F, P \in \mathcal{M}_n} \mathds{1}_\mathcal{C}(F) \sum_{c \in \mathbb{F}_q} \psi\left((f_0 - p_0)a\right)~.
\end{eqnarray*}
By the orthogonality relations we have
\begin{equation*}
	X = \frac{1}{q^n} \sum_{F, P \in \mathcal{M}_n} \mathds{1}_\mathcal{C}(F) q\mathds{1}_{f_0 = p_0} = \frac{1}{q^{n-1}} \sum_{F, P \in \mathcal{M}_n} \mathds{1}_\mathcal{C}(F) \mathds{1}_{f_0 = p_0}.
\end{equation*}
Summing over $c = f_0 = p_0 \in \mathbb{F}_q$, we see that
\begin{equation}\label{main_term_product}
	X =  \frac{1}{q^{n-1}} \sum_{c \in \mathbb{F}_q}  \left( \sum_{\substack{F \in \mathcal{M}_n \\ f_0 = c}} \mathds{1}_\mathcal{C}(F) \right) \left( \sum_{\substack{P \in \mathcal{M}_n \\ p_0 = c}}1 \right)
\end{equation} 
Note that the middle sum is exactly 
\begin{equation}\label{monic_c_product}
	\sum_{\substack{F \in \mathcal{M}_n \\ f_0 = c}} \mathds{1}_\mathcal{C}(F) = \mathds{1}_{\mathcal{C}_0}(c)\prod_{1 \leq i \in \mathcal{J}}(q-N_i)~,
\end{equation}
so substituting this into \eqref{main_term_product} gives
\begin{equation}\label{take_prod_out}
	X = \frac{1}{q^{n-1}}\prod_{1 \leq i \in \mathcal{J}}(q-N_i)\sum_{c \in \mathbb{F}_q}  \mathds{1}_{\mathcal{C}_0}(c)\sum_{\substack{P \in \mathcal{M}_n \\ p_0 = c}}1
\end{equation}
When $c = 0$, we have  $ \sum_{\substack{P \in \mathcal{M}_n \\ P_0 = c}}1 = 0$. Otherwise, by the Prime Polynomial Theorem in arithmetic progressions we have $\sum_{\substack{P \in \mathcal{M}_n \\ P_0 = c}}1 = \frac{1}{q-1}\pi_q(n) + O\left(q^{n/2}\right)$.
If $0 \in \mathcal{I}$, then the only $c$ for which $ \mathds{1}_{\mathcal{C}_0}(c) \neq 0$ is $c = a_0$. Thus in this case
\begin{eqnarray*}
X
	&=& \frac{1}{q^{n-1}}\prod_{1 \leq i \in \mathcal{J}}(q-N_i)\sum_{\substack{P \in \mathcal{M}_n \\ p_0 = a_0}}1 \\
	&=& \frac{1}{q^{n-1}}\prod_{1 \leq i \in \mathcal{J}}(q-N_i)\left( \frac{1}{q-1}\pi_q(n) + O(q^{n/2}) \right)
\end{eqnarray*}
So we have shown that in this case
\begin{equation}\label{0_in_i_final}
	X
	= \left( \frac{1}{q^{n-1}(q-1)} \prod_{i \in \mathcal{J}} (q - N_i) \right) \pi_q(n) + O(q^{n/2})
\end{equation}

In the case where $0 \in \mathcal{J}$, we have that $\mathds{1}_{\mathcal{C}_0}(c) \neq 0$ for $c \in \mathbb{F}_q \backslash S_0$. So in this case
\begin{equation}\label{sum_over_not_s0}
	\sum_{c \in \mathbb{F}_q}  \mathds{1}_{\mathcal{C}_0}(c)\sum_{\substack{P \in \mathcal{M}_n \\ p_0 = c}}1
	= \sum_{c \in \mathbb{F}_q \backslash S_0} \sum_{\substack{P \in \mathcal{M}_n \\ p_0 = c}}1 =
	\pi_q(n) - \sum_{c \in S_0} \sum_{\substack{P \in \mathcal{M}_n \\ p_0 = c}}1
\end{equation}
By assumption $0 \notin S_0$, thus
\begin{equation*}
	  \sum_{c \in S_0}\sum_{\substack{P \in \mathcal{M}_n \\ p_0 = c}}1 = \frac{N_0}{q-1}\pi_q(n) + O\left(N_0q^{n/2}\right)
\end{equation*}
Plugging this into \eqref{sum_over_not_s0} gives
\begin{equation}\label{sum_0_in_j}
	\sum_{c \in \mathbb{F}_q}  \mathds{1}_{\mathcal{C}_0}(c)\sum_{\substack{P \in \mathcal{M}_n \\ p_0 = c}}1
		= \frac{q - N_0 - 1}{q-1}\pi_q(n) + O\left(N_0q^{n/2}\right),
\end{equation}
and combining \eqref{sum_0_in_j} and \eqref{take_prod_out} results in
\begin{align}\label{0_in_j_final}
	X =&~
		\left(\frac{1}{q^{n-1}}\prod_{1 \leq i \in \mathcal{J}}(q-N_i)\right) \cdot  \left(\frac{q - N_0 - 1}{q-1}\pi_q(n) + O\left(N_0q^{n/2}\right)\right) \\
\nonumber	=&~ \left(\frac{q-N_0-1}{q^{n-1}(q-1)}\prod_{1 \leq i \in \mathcal{J}}(q-N_i) \right) \pi_q(n) + O\left(N_0 q^{n/2}\right)
\end{align}

Writing 
\begin{equation}\label{sigma_def}
\mathfrak{S} = 
\left\{
	\begin{array}{ll}
		\frac{1}{q^{n-1}(q-1)} \prod_{i \in \mathcal{J}} (q - N_i) & \mbox{if } 0 \in \mathcal{I} \\
		\\
		\frac{q-N_0-1}{q^{n-1}(q-1)}\prod_{1 \leq i \in \mathcal{J}}(q-N_i) & \mbox{if } 0 \in \mathcal{J}
	\end{array}
\right.
\end{equation}
we have shown in \eqref{0_in_i_final} and \eqref{0_in_j_final} that 
\begin{equation}\label{main_term_final}
	\left| \frac{1}{q^n}\sum_{a \in \mathbb{F}_q} \hat{F}_{q,n}(a/T)f(-a/T) - \mathfrak{S} \cdot \pi_q(n) \right| = 
	\left| X - \mathfrak{S} \cdot \pi_q(n) \right| = 
	 O\left(q^{n/2 + 1}\right)
\end{equation}

\subsection{Conclusion}\label{conclusion_sub}

Recall that
\begin{eqnarray*}
	\sum_{P \in \mathcal{C}} 1 &=& \frac{1}{q^n}\sum_{F \in \mathcal{M}_n}\hat{F}_{q,n}(\theta_F)f(-\theta_F) \\
		&=&  \frac{1}{q^n}\sum_{F \in \mathcal{S}_1 \cup \mathcal{S}_2 \cup \mathcal{S}_3}\hat{F}_{q,n}(\theta_F)f(-\theta_F)~,
\end{eqnarray*}
where $\mathcal{S}_1$, $\mathcal{S}_2$, $\mathcal{S}_3$ provide a partition of $\mathcal{M}_n$ defined in \eqref{m_n_partition}.
Thus 
\begin{equation*}
	\left|\left(\sum_{P \in \mathcal{C}} 1\right) - \frac{1}{q^n}\sum_{F \in \mathcal{S}_1}\hat{F}_{q,n}(\theta_F)f(-\theta_F)\right|
		=  \left| \frac{1}{q^n}\sum_{F \in \mathcal{S}_2}\hat{F}_{q,n}(\theta_F)f(-\theta_F) + \frac{1}{q^n}\sum_{F \in \mathcal{S}_3}\hat{F}_{q,n}(\theta_F)f(-\theta_F) \right|~.
\end{equation*}
Bounds for 
\begin{equation*}\label{error_term_by_s1}
		\left|\left(\sum_{P \in \mathcal{C}} 1\right) - \frac{1}{q^n}\sum_{F \in \mathcal{S}_1}\hat{F}_{q,n}(\theta_F)f(-\theta_F)\right|
\end{equation*}
are given in \eqref{eq_error_bound_smalli} and \eqref{eq_error_bound_bigi}. For the main term, we have shown in \eqref{main_term_final} that
\begin{equation*}
	\left| \frac{1}{q^n}\sum_{F \in \mathcal{S}_1}\hat{F}_{q,n}(\theta_F)f(-\theta_F) - \mathfrak{S} \cdot \pi_q(n) \right| =  O\left(q^{n/2 + 1}\right)~,
\end{equation*}
with $\mathfrak{S}$ defined as in \eqref{sigma_def}.
Thus, by the triangle inequality,
\begin{equation}\label{eq_main_term}
	\left| \left(\sum_{P \in \mathcal{C}} 1\right)  -  \mathfrak{S} \cdot \pi_q(n)  \right| \leq
		\left| \left(\sum_{P \in \mathcal{C}} 1\right)  - \frac{1}{q^n}\sum_{F \in \mathcal{S}_1}\hat{F}_{q,n}(\theta_F)f(-\theta_F) \right| + O\left(q^{n/2 + 1}\right)~.
\end{equation}
Writing $s = \sqrt{(1 - \varepsilon)/(1 - 2\varepsilon)}\sqrt{n}$ and $m_{n,I,\varepsilon} = \min\{n/I,s\}$, plugging \eqref{eq_error_bound_smalli} into \eqref{eq_main_term} yields
\begin{eqnarray*}\left| \left(\sum_{P \in \mathcal{C}} 1\right)  - \mathfrak{S} \cdot \pi_q(n) \right| &\leq&  \alpha(n-I)q^{n - \frac{1}{2}\lfloor \frac{n}{2} \rfloor} + q^{n-I}q^{-\left(1 - 2\varepsilon + o(1)\right)m_{n,I,\varepsilon}}  + O\left(q^{n/2 + 1}\right) \\
	&=&  \left(\alpha(n-I) + o(1)\right)q^{n - \frac{1}{2}\lfloor \frac{n}{2} \rfloor} + q^{n-I}q^{-\left(1 - 2\varepsilon + o(1)\right)m_{n,I,\varepsilon}}~,
\end{eqnarray*} 
when $I = o\left(n/\log(n)\right)$.
If we have larger $I$, we assume that $y = \frac{I \cdot s}{n} > 1$ and that $n$ is sufficiently large in terms of $\varepsilon$ and $\tau$. In this case, we use the bound given in  \eqref{eq_error_bound_bigi} together with \eqref{eq_main_term} in order to obtain
\begin{eqnarray*}
	\left| \left(\sum_{P \in \mathcal{C}} 1\right)  - \mathfrak{S} \cdot \pi_q(n) \right| 
	&\leq&  \alpha(n-I)q^{n - \frac{1}{2}\lfloor \frac{n}{2} \rfloor} + q^{n-I}q^{-(1 - 2\varepsilon)n/I + 4 - 3 \varepsilon + B_{q,\varepsilon,\tau,y}}  + O\left(q^{n/2 + 1}\right) \\
	&=& \left(\alpha(n-I) + o(1)\right)q^{n - \frac{1}{2}\lfloor \frac{n}{2} \rfloor} + q^{n-I}q^{-(1 - 2\varepsilon)n/I + 4 - 3 \varepsilon + B_{q,\varepsilon,\tau,y}}~,
\end{eqnarray*} 
with $B_{q,\varepsilon,\tau,y}$ tending to zero as $q$ grows to infinity.
This completes the proof of the theorem.

\section{Discussion}\label{discussion_section}
In the introduction we defined a set $\mathcal{C}$ and stated how many primes one might expect $\mathcal{C}$ to contain. For convenience, we defined $\mathcal{C}$ to be
\begin{equation*}
	\mathcal{C} = \left\{T^n + \sum_{i = 0}^{n-1}b_iT^i : b_i = a_i~ \forall i \in \mathcal{I}, b_j \not \in S_j~ \forall j \in \mathcal{J} \right\}~,
\end{equation*}
where $\mathcal{I} \uplus \mathcal{J}$ be a partition of $\{0,\dots,n-1\}$,  $a_0 \neq 0$ if $0 \in \mathcal{I}$, and  $0 \notin S_0$ if $0 \in \mathcal{J}$. Write $I = \# \mathcal{I}$, and for every $j \in \mathcal{J}$ write $N_j =  \# S_j$.
The number of primes to be expected in $\mathcal{C}$ is $\mathfrak{S} \cdot \pi_q(n)$, with 
\begin{equation*}
\mathfrak{S} = 
\left\{
	\begin{array}{ll}
		 \frac{\prod_{j \in \mathcal{J}} (q - N_j)}{q^{n-1}(q-1)}  & \mbox{if } 0 \in \mathcal{I}\\
		 \\
		 \frac{(q-1-N_0)\prod_{0 < j \in \mathcal{J}} (q - N_j)}{q^{n-1}(q-1)} & \mbox{if } 0 \in \mathcal{J}~.
	\end{array}
\right.
\end{equation*}
The asymptotics are indeed what one might expect: Note that there are $q-N_i$ options for every index $i \in \mathcal{J}$, so $\# \mathcal{C} = \prod_{i \in \mathcal{J}}(q - N_i)$. We expect the proportion of  primes in $\mathcal{C}$ to be similar to that in all of the monic polynomials, up to a correction factor due to the coefficient $b_0$. We think of $\mathfrak{S}$ as $\mathfrak{S} = \frac{\# \mathcal{C}}{q^n} \cdot R$, with $R$ being a correction factor. If $0 \in \mathcal{I}$, then since $a_0 \neq 0$, the probability of being prime increases by a factor of $\frac{q}{q-1}$, hence in this case
\begin{equation*}
\mathfrak{S} = \frac{\# \mathcal{C}}{q^n} \cdot R = \frac{\# \mathcal{C}}{q^n} \cdot \frac{q}{q-1} = \frac{1}{q^{n-1}(q-1)} \prod_{i \in \mathcal{J}} (q - N_i).
\end{equation*}
In the other case where $0 \in \mathcal{J}$, having assumed $0 \notin S_0$ the correction factor will be \mbox{$R = \frac{q - N_i - 1}{q-1} \cdot \frac{q}{q-1}$}, so in this case
\begin{equation*}
\mathfrak{S} = \frac{\# \mathcal{C}}{q^n} \cdot R = \frac{\# \mathcal{C}}{q^n} \cdot \frac{q - N_i - 1}{q-1} \cdot \frac{q}{q-1} = \frac{q - N_0 - 1}{q^{n-1}(q-1)} \prod_{1 \geq i \in \mathcal{J}} (q - N_i) .
\end{equation*}

\bibliographystyle{amsplain}
\providecommand{\bysame}{\leavevmode\hbox to3em{\hrulefill}\thinspace}
\providecommand{\MR}{\relax\ifhmode\unskip\space\fi MR }
\providecommand{\MRhref}[2]{%
  \href{http://www.ams.org/mathscinet-getitem?mr=#1}{#2}
}
\providecommand{\href}[2]{#2}

\end{document}